\documentclass[11pt]{article}

\usepackage{verbatim}

\usepackage{amsfonts}

\usepackage{amsthm}

\usepackage{amssymb}

\usepackage{amsmath}

\usepackage{mathabx}

\usepackage{enumerate}

\usepackage[all]{xy}

\usepackage{graphicx}

\usepackage{hyperref}

\newcommand{\N}{\mathbb{N}}

\newcommand{\Z}{\mathbb{Z}}

\newcommand{\R}{\mathbb{R}}

\newcommand{\C}{\mathbb{C}}

\newcommand{\Manoa}{M\=anoa}

\newcommand{\Hawaii}{Hawai\kern.05em`\kern.05em\relax i}

\numberwithin{equation}{section}

\theoremstyle{plain}
\newtheorem{theorem}{Theorem}[section]
\newtheorem{lemma}[theorem]{Lemma}
\newtheorem{corollary}[theorem]{Corollary}
\newtheorem{proposition}[theorem]{Proposition}

\newtheorem{definition-theorem}[theorem]{Definition / Theorem}

\newtheorem*{conjecture*}{Conjecture}
\newtheorem*{theorem*}{Theorem}

\theoremstyle{definition}
\newtheorem{definition}[theorem]{Definition}
\newtheorem{example}[theorem]{Example}

\newtheorem{remark}[theorem]{Remark}

\newtheorem*{example*}{Example}  
\newtheorem*{remark*}{Remark}

\title{Cartan subalgebras in uniform Roe algebras}
\author{Stuart White\thanks{Research partially supported by EPSRC (EP R025061/1) and an Alexander Humboldt foundation fellowship.} \and Rufus Willett\thanks{Research partially supported by US NSF grants DMS 1401126 and DMS 1564281}}

\begin{document}

\maketitle

\begin{abstract}In this paper we study structural and uniqueness questions for Cartan subalgebras of uniform Roe algebras. We characterise when an inclusion $B\subseteq A$ of $\mathrm{C}^*$-algebras is isomorphic to the canonical inclusion of $\ell^\infty(X)$ inside a uniform Roe algebra $C^*_u(X)$ associated to a metric space of bounded geometry. We obtain uniqueness results for `Roe Cartans' inside uniform Roe algebras up to automorphism when $X$ coarsely embeds into Hilbert space, and up to inner automorphism when $X$ has property A.
\end{abstract}

\section{Introduction}

\renewcommand*{\thetheorem}{\Alph{theorem}}

The aim of this paper is to study Cartan subalgebras in uniform Roe algebras, and in particular to what extent the `standard' Cartan subalgebra is unique.  Roe algebras associated to metric spaces were introduced in \cite{Roe:1993lq} for their connections to (higher) index theory and the associated applications to manifold topology and geometry \cite{Roe:1996dn,Yu:200ve}.  The uniform variant of the Roe algebra has since been fairly extensively studied for its own sake, and provides an interesting bridge between coarse geometry and $C^*$-algebra theory.

It is natural to ask how much of the information about a metric space is remembered by, or can be recovered from, the associated uniform Roe algebra; this line of research was initiated by \v{S}pakula and the second author in \cite{Spakula:2011bs}.  Such \emph{rigidity} questions are strongly motivated by the coarse Baum-Connes conjecture and its variants \cite{Higson:1995fv,Yu:1995bv,Spakula:2009tg,Engel:2014tx}. Roughly speaking these conjectures predict that the analytic K-theory of the (uniform) Roe algebra provides a faithful model for the large scale algebraic topology of the underlying space, i.e.\ they postulate that on the level of $K$-theory, $C^*_u(X)$ retains all relevant information about $X$.  If these conjectures have a positive answer, one can apply powerful analytic tools (positivity and the spectral theorem) to the study of $X$, and thus deduce important consequences in topology and geometry.    This latter motivation has been made particularly stark by recent results of Braga and Farah \cite{Braga:2018dz}, who show that possible failure of rigidity is intimately tied to the existence of so-called ghost operators that are also known to cause problems for the coarse Baum-Connes conjecture (see \cite[Section 6]{Higson:2002la} and \cite[Sections 5-6]{Willett:2010ud}).

\smallskip

On the other hand, Cartan subalgebras have been present in the study of operator algebras since the foundational work of Murray and von Neumann. Indeed, the  prototypical example of a Cartan subalgebra arises from Murray and von Neumann's group measure space construction. Reminiscent of the semidirect product construction in group theory, given a group $G$ acting by non-singular transformations on a measure space $(X,\mu)$, one constructs a single von Neumann algebra $L^\infty(X,\mu)\rtimes G$ containing a copy of $L^\infty(X,\mu)$ and so that the induced action of $G$ on $L^\infty(X,\mu)$ is by inner automorphisms. Here $L^\infty(X,\mu)$ is a Cartan subalgebra of $L^\infty(X,\mu)\rtimes G$. 

Abstracting the properties of the inclusion $L^\infty(X,\mu)\subseteq L^\infty(X,\mu)\rtimes G$, Vershik defined the concept of a Cartan subalgebra \cite{Vershik:1971aa}, and this was extensively studied by Feldman and Moore \cite{Feldman:1977tx,Feldman:1977rd} who showed that these subalgebras correspond to (twisted) measured orbit equivalence relations.  Thus Cartan subalgebras provide an operator algebraic framework for the study of dynamical systems.  Moreover, a  major step in understanding the range of possible group actions giving rise to the same crossed product algebra involves classifying Cartan subalgebras.

Voisculescu famously showed free group factors have no Cartan subalgebras \cite{Voiculescu:1996ux}, while in the uniqueness direction, a celebrated theorem of Connes, Feldmann and Weiss shows that injective von Neumann algebras with separable predual have unique Cartan subalgebras up to automorphism \cite{Connes:1981aa}, i.e. if $A,B$ are Cartan subalgebras in an injective von Neumann algebra $M$, then there is an automorphism $\alpha$ of $M$ with $\alpha(A)=B$.  In the injective setting, it will rarely be the case that this automorphism can be taken to be inner (cf.\ \cite[Theorem 7]{Feldman:1977rd}), and a major breakthrough was made by Ozawa and Popa who gave the first example of a II$_1$ factor with a unique Cartan subalgebra up to inner automorphism in \cite{Ozawa:2010sz}. Subsequently Popa's deformation-rigidity theory has been used to produce a number of striking uniqueness and non-uniqueness results for Cartan subalgebras in von Neumann factors: see for example  \cite{Ozawa:2010el,Vaes:2010aa,Chifan:2013ji,Popa:2014jq,Popa:2014cu}.

Corresponding notions have been developed in the setting of $\mathrm{C}^*$-algebras. Building on Kumjian's much earlier notion of a $\mathrm{C}^*$-diagonal \cite{Kumjian:1986aa}, Renault defined a Cartan pair in \cite{Renault:2008if}, showing that any such pair is isomorphic to the inclusion $C_0(G^{(0)})\subseteq C^*_r(G,\Sigma)$ of the $C_0$-functions on the unit space $G^{(0)}$ of a twisted, \'{e}tale, topologically principal groupoid $(G,\Sigma)$ into the associated twisted groupoid $C^*$-algebra.  Such a decomposition is particularly useful in the nuclear case, as it implies that the universal coefficient theorem of Rosenberg and Schochet holds \cite{Barlak:2017aa}.  Recently, there has been growing interest in studying general existence and uniqueness questions for Cartan subalgebras in $C^*$-algebras \cite{Li:2017aa,Carlsen:2017aa}. In contrast to the von Neumann algebraic setting, even very elementary $\mathrm{C}^*$-algebras such as dimension drop algebras and UHF-algebras have multiple Cartan subalgebras \cite{Blackadar:1990aa,Barlak:2017ab}.  One key difference is that separable measure spaces are readily classified, while compact metrisable spaces are not. Indeed every non-atomic Cartan subalgebra in a von Neumann algebra with separable predual is abstractly isomorphic to $L^\infty([0,1])$. In the references above the spectrum is used to distinguish Cartan subalgebras.  So in the setting of $\mathrm{C}^*$-algebras one should really only attempt to classify Cartan subalgebras with a specified spectrum.

The key example relevant to this paper is the canonical Cartan subalgebra in a uniform Roe algebra. If $\Gamma$ is a countable group, the uniform Roe algebra $C^*_u(\Gamma)$ is the reduced group $\mathrm{C}^*$-algebra crossed product $\ell^\infty(\Gamma)\rtimes_r \Gamma$, where the action is by left translation.  Thus uniform Roe algebras have a mixed $C^*$-algebraic (from the reduced crossed product) and von Neumann algebraic (from $\ell^\infty$) identity, which suggests they are a good candidate for pushing uniqueness of Cartan results into the $C^*$-world. The subalgebra $\ell^\infty(\Gamma)$ provides a canonical Cartan subalgebra inside $C^*_u(\Gamma)$. More generally, when $X$ is a metric space of bounded geometry, $\ell^\infty(X)$ is a Cartan subalgebra of $C^*_u(X)$; this corresponds to the description of $C^*_u(X)$ as a groupoid $C^*$-algebra due to Skandalis, Tu, and Yu, \cite{Skandalis:2002ng}.

\smallskip

Our aim in this paper is to study the following questions.
\begin{itemize}
\itemsep-.1cm 
\item What form can general Cartan subalgebras in a uniform Roe algebra take? This could mean what isomorphism type as an abstract $C^*$-algebra, or it could mean the more refined spatial theory of how a Cartan subalgebra can be represented on $\ell^2(X)$.
\item When does an abstract Cartan pair $B\subseteq A$ come from a uniform Roe algebra?
\item To what extent is the canonical Cartan subalgebra in a uniform Roe algebra unique? Here uniqueness might mean up to automorphism or more strongly up to inner automorphism, and might refer to uniqueness among some class of Cartan subalgebras satisfying additional conditions.
\end{itemize}

We address the first question in Section \ref{gen cart sec}, where we work in the generality of $\mathrm{C}^*$-algebras between the compact and bounded operators on a Hilbert space.  Specialising our results to uniform Roe algebras we obtain the following proposition (which is a combination of the more general statements Theorem \ref{cart com the} and Proposition \ref{non sep}).

\begin{proposition}\label{IntroA}
Let $X$ be a countably infinite metric space of bounded geometry. Then any Cartan subalgebra $B\subseteq C^*_u(X)$ is non-separable and contains a complete family of orthonormal rank one projections for $\ell^2(X)$.
\end{proposition}

While Cartan subalgebras in uniform Roe-algebras must be non-separable, they do not have to be abstractly isomorphic to $\ell^\infty$, and even relatively straightforward metric spaces admit Cartan subalgebras with exotic spectra. This is the subject of Section \ref{S3}.

Thus, and as expected in the $\mathrm{C}^*$-setting, we must impose additional structure such as spectral data in order to recognise the canonical Cartan subalgebra amongst all possible Cartan subalgebras of a uniform Roe algebra. We explore this in Section \ref{abs cart sec}, abstracting the following key features of the inclusion $\ell^\infty(X)\subseteq C^*_u(X)$ into the concept of a \emph{Roe Cartan pair} (see Definition \ref{roe cartan def}): 
\begin{itemize}
\itemsep-.1cm 
\item containment of the compacts as an essential ideal;
\item the Cartan subalgebra is abstractly isomorphic to $\ell^\infty(\mathbb N)$;
\item countable generation of the containing algebra over the subalgebra (``co-separability'').
\end{itemize}
Such Cartan pairs can only arise from canonical Cartan subalgebras in uniform Roe algebras. 

\begin{theorem}\label{cart abs}
Let $B\subseteq A$ be a Roe Cartan pair.   Then there exists a bounded geometry metric space $Y$ such that for any irreducible and faithful representation of $A$ on a Hilbert space $H$ there is a unitary isomorphism $v:\ell^2(Y)\to H$ such that 
$$
v^*Bv=\ell^\infty(Y) \quad \text{and}\quad v^*Av=C^*_u(Y).
$$
\end{theorem}

When the algebra $A$ above is already a uniform Roe algebra associated to a metric space $X$, then it is natural to ask how $X$ and the space $Y$ produced by the previous theorem are related.  Using the very recent preprint \cite{Braga:2019wv}, we have the following corollary. 

\begin{corollary}\label{IntroC}
Let $X$ be a bounded geometry metric space that coarsely embeds into a Hilbert space.  If $B\subseteq C^*_u(X)$ is a Roe Cartan pair, then the bounded geometry metric space associated to this pair by Theorem \ref{cart abs} is coarsely equivalent to $X$.
\end{corollary}

The hypotheses of the above theorem apply broadly.  Every exact group (in the sense of Kirchberg and Wassermann \cite{Kirchberg:1999ss}) coarsely embeds into a Hilbert space by \cite[Theorem 2.2]{Yu:200ve}. Exactness of a group $\Gamma$ is characterised coarsely through Yu's property A for the associated metric space $X$, which is also equivalent to nuclearity of $C^*_u(X)$ (see \cite{Ozawa:2000th}).  The class of exact groups is very large, including for example all linear groups \cite{Guentner:2005xr}, all groups with finite asymptotic dimension \cite{Higson:2000dp}, and all amenable groups; see \cite{Willett:2009rt} for a survey.  In fact by \cite[Theorem 2]{Osajda:2014ys}, the class of groups which coarsely embed into a Hilbert space is strictly larger than the exact groups, and the only groups that are known not to coarsely embed into Hilbert space are the so-called \emph{Gromov monster} groups whose Cayley graphs contain expanders \cite{Gromov:2003gf,Arzhantseva:2008bv,Osajda:2014ys}.

\smallskip
We now turn to uniqueness results for Cartan subalgebras of Roe algebras. Any Cartan subalgebra of a uniform Roe algebra that is conjugate by an automorphism to the canonical Cartan must be a Roe Cartan, so we only ask for uniqueness for Roe Cartans. Using results of Whyte \cite{Whyte:1999uq}, we can obtain uniqueness up to automorphism whenever the space $X$ coarsely embeds into Hilbert space and is non-amenable in the sense of Block and Weinberger \cite{Block:1992qp} (when $X$ is the metric space associated to a finitely generated group, non-amenability is precisely failure of amenability of the group in the usual sense \cite[Chapter 3]{Roe:2003rw}).  In particular, the following corollary (proved in Section \ref{inn sec}) applies to examples like non-abelian free groups, non-elementary word hyperbolic groups, and lattices in higher rank semi-simple Lie groups.

\begin{corollary}\label{IntroD}
Let $X$ be a bounded geometry metric space that coarsely embeds into a Hilbert space, and is non-amenable.  Let $B\subseteq C^*_u(X)$ be a Roe Cartan subalgebra.   Then there is a $^*$-automorphism $\alpha$ of $C^*_u(X)$ such that $\alpha(\ell^\infty(X))=B$.
\end{corollary}

Finally we turn to the strong form of uniqueness up to inner automorphism.   Theorem \ref{c*-alg the} is the central result of the paper. It uses both Proposition \ref{IntroA} and Theorem \ref{cart abs} above as ingredients in its proof. Other key ingredients include the rigidity results from \cite{Spakula:2011bs}, a recent criterion for detecting when an operator lies in a uniform Roe algebra under the hypothesis of property A due to \v{S}pakula and Zhang \cite{Spakula:2018aa}  (which builds on work of \v{S}pakula and Tikuisis \cite{Spakula:2017aa}), the operator norm localisation property of \cite{Chen:2008so}, and results of Braga and Farah \cite{Braga:2018dz}.  

\begin{theorem}\label{c*-alg the}
Let $X$ be a bounded geometry metric space with property A.  Let $B\subseteq C^*_u(X)$ be a Roe Cartan subalgebra.  Then there is a unitary operator $u\in C^*_u(X)$ such that $uBu^*=\ell^\infty(X)$.
\end{theorem}

Theorem \ref{c*-alg the} gives a stronger conclusion than Corollary \ref{IntroD} in the property A case.  However, as noted above there are examples where Corollary \ref{IntroD} applies and Theorem \ref{c*-alg the} does not: see for example \cite{Osajda:2014ys}.

It seems plausible to us that Theorem \ref{c*-alg the} will fail without some assumption on $X$, due to the well-known exotic analytic properties of uniform Roe algebras outside of the property A setting; see for example \cite{Sako:2012kx} and \cite{Roe:2013rt}. We would be very interested in any progress towards the construction of exotic examples, or in showing that they cannot exist.

\medskip

\noindent\textbf{Acknowledgements.}  We would like to thank Bruno Braga, Sel\c{c}uk Barlak, Ilijas Farah, Kang Li, J\'an \v{S}pakula, and Aaron Tikuisis for their helpful conversations, insights and comments on earlier versions of this paper. Part of the underlying research for this paper was undertaken during the conference Noncommutative Dimension Theories at the University of \Hawaii~at \Manoa~in 2015, and the IRP on Operator Algebras: Dynamics and Interactions at the CRM in Barcelona in 2017.  We would like to thank the (other) organisers of those meetings, and the funding bodies for supporting the meetings.

We would also like to thank the referees for their helpful reports.
\numberwithin{theorem}{section}

\section{Cartan subalgebras of $\mathrm{C}^*$-algebras containing the compacts}\label{gen cart sec}

Our aim in this section is to prove some general structural results about Cartan subalgebras in $C^*$-algebras that contain the compact operators.  We begin by recalling the definition of a Cartan subalgebra from \cite{Renault:2008if}.

\begin{definition}\label{cart def}
Let $A$ be a $C^*$-algebra.  A \emph{Cartan subalgebra} of $A$ is a $C^*$-subalgebra $B\subseteq A$ such that:
\begin{enumerate}[(i)]
\itemsep-.1cm 
\item $B$ is a maximal abelian self-adjoint subalgebra (MASA) of $A$;
\item $B$ contains an approximate unit\footnote{We will mainly be interested in the case that $A$ is unital, in which case condition (ii) is automatic: indeed condition (i) implies that $B$ contains the unit of $A$.}  for $A$ 
\item the \emph{normaliser} of $B$ in $A$, defined as
$$
\mathcal{N}_A(B):=\{a\in A\mid aBa^*\cup a^*Ba\subseteq B\}
$$
generates $A$ as a $C^*$-algebra;
\item there is a faithful conditional expectation $E:A\to B$.
\end{enumerate}
A \emph{Cartan pair} is a nested pair $B\subseteq A$ of $C^*$-algebras such that $B$ is a Cartan subalgebra of $A$.
\end{definition}

For later purposes we make the following definition.
\begin{definition}\label{cosepdef}
We say that a Cartan subalgebra $B$ of $A$ is \emph{co-separable} if there is a countable subset $S$ of $A$ (or equivalently of $\mathcal N_A(B)$) such that $A=C^*(S,B)$.\end{definition}

We need the following routine fact.
\begin{lemma}\label{new lemma}
Let $A\subseteq\mathcal B(H)$ be a concrete $C^*$-algebra containing the compact operators on $H$, and let $B\subseteq A$ be a maximal abelian subalgebra.  Then any minimal projection in $B$ is rank one.
\end{lemma}
\begin{proof}
If $p\in B$ is minimal and not rank one, then there exists a rank one projection $q\in A$ with $q\leq p$.  However, $q$ commutes with $B=pB\oplus (1-p)B$, a contradiction.
\end{proof}

\begin{lemma}\label{comp lem}
Let $A\subseteq \mathcal{B}(H)$ be a concrete $C^*$-algebra containing the compact operators on $H$.  Let $B\subseteq A$ be a maximal abelian subalgebra, equipped with a conditional expectation $E:A\to B$.  Then for any compact operator $a\in A$, $E(a)$ is also compact.  
\end{lemma}

\begin{proof}
It suffices to show that $E(e)$ is compact for any rank one projection $e$ on $H$, which we fix from now on.  First, we establish the following claim, called ($*$) in the rest of the proof: there cannot exist $\lambda>0$ such that for any $N\in \N$ there are positive and mutually orthogonal contractions $b_1,...,b_N$ in $B$ such that $\|b_iE(e)b_i\|\geq \lambda$ for each $i$.  Indeed, if such a $\lambda>0$ exists, then find $b_1,...,b_N$ with the properties above.  Let $\text{Tr}:\mathcal{B}(H)_+\to [0,\infty]$ be the canonical unbounded trace.  Then we have that 
\begin{equation}\label{small trace}
\text{Tr}\Big(\sum_{i=1}^N b_ieb_i\Big)=\Bigg|\text{Tr}\Big(\sum_{i=1}^N b_i^2e\Big)\Bigg|\leq \Bigg\|\sum_{i=1}^Nb_i^2\Bigg\|\text{Tr}(e)=1,
\end{equation}
where the last inequality follows as mutual orthogonality of the $b_i$ gives $\|\sum_{i=1}^Nb_i^2\|=\sup_{i=1}^N \|b_i^2\|$, and this is at most one as each $b_i$ is a contraction.  
On the other hand, using that $E$ is a conditional expectation (so in particular contractive) and that the $b_i$ are in $B$, we have that 
\begin{equation}
\|b_ieb_i\|\geq \|E(b_ieb_i)\|=\|b_iE(e)b_i\|\geq \lambda
\end{equation}
for each $i$.  Combining this with \eqref{small trace} and using $\|\cdot\|_1$ for the trace norm, we have
\begin{equation}
1\geq \text{Tr}\Big(\sum_{i=1}^N b_ieb_i\Big)=\sum_{i=1}^N \|b_ieb_i\|_1\geq \sum_{i=1}^N \|b_ieb_i\|\geq N\lambda
\end{equation}
(the penultimate inequality follows as the trace norm is always at least as big as the operator norm).  As $N$ was arbitrary, this is impossible, proving claim ($*$).

We next claim that for any $\lambda>0$, the intersection of the spectrum of $E(e)$ and $[\lambda,\infty)$ must be finite.  Indeed, if not, then fix $\lambda>0$ such that the intersection of the spectrum of $E(e)$ with $[\lambda,\infty)$ is infinite.  For any $N$, there are continuous functions $\phi_1,...,\phi_N:\R\to [0,1]$ supported on $[\lambda,\infty)$, with mutually disjoint supports, and with the property that each $\phi_i$ attains the value $1$ somewhere on the intersection of the spectrum of $E(e)$ and $[\lambda,\infty)$.  Setting $b_i:=\phi_i(E(e))$ the functional calculus gives us that the $b_i$ are positive, mutually orthogonal contractions with $\|b_iE(e)b_i\|\geq \lambda$ for each $i$ and so we have contradicted claim ($*$).

Thus the spectrum of $E(e)$ is a countable subset of $[0,\infty)$, and the only possible limit point is $0$. Given $\lambda>0$ in this spectrum, let $p:=\chi_{\{\lambda\}}(E(e))\in B$ be the associated spectral projection.  Suppose by way of reaching a contradiction that $p$ has infinite rank. By Lemma \ref{new lemma}, $p$ is not a minimal projection in $B$ so has a proper subprojection $p_1\in B$. By replacing $p_1$ with $p-p_1$ if necessary we may assume $p_1$ is also infinite rank.  Repeating this argument we obtain a strictly decreasing infinite sequence $p\geq p_1\geq p_2\geq \cdots $ of infinite rank projections in $B$.  Set $b_i:=p_i-p_{i-1}$.  Then for any $i$, we have 
\begin{equation}
\|b_iE(e)b_i\|\geq \frac{1}{\lambda}\|b_ipb_i\|=\frac{1}{\lambda}.
\end{equation}
This contradicts claim ($*$). Therefore $p$ is finite rank, and hence $E(e)$ is compact.
\end{proof}

\begin{lemma}\label{onb lem}
Suppose that $A\subseteq \mathcal{B}(H)$ is a concrete $C^*$-algebra containing the compact operators on $H$.  Let $B\subseteq A$ be a Cartan subalgebra.  Then $B$ contains a complete orthogonal set of rank one projections.
\end{lemma}

\begin{proof}
Write $E:A\to B$ for the faithful conditional expectation that comes with the fact that $B$ is a Cartan subalgebra of $A$, and let $e$ be a rank one projection.  Then $E(e)$ is compact by Lemma \ref{comp lem}, and non-zero as $E$ is faithful.  It follows from the spectral theorem that $B$ contains a non-zero finite rank projection, and thus a minimal non-zero finite rank projection, say $q$, which must be rank one by Lemma \ref{new lemma}.

Let now $S$ be the collection of all rank one projections in $B$, which is non-empty by the above argument.  As $B$ is commutative, the projections in $S$ are all mutually orthogonal, and thus the sum $p:=\sum_{q\in S} q$ converges strongly to a non-zero projection.  Note that as $p$ is a strong limit of operators in $B$, it commutes with everything in $B$.  We claim that in fact $p$ commutes with everything in the normaliser of $B$ in $A$.  Indeed, if not, there exists $a\in \mathcal{N}_A(B)$ such that $pa(1-p)\neq 0$.  The definition of $p$ thus gives a rank one projection $q$ in $B$ such that $qa(1-p)\neq 0$.  Hence $(1-p)a^*qa(1-p)\neq 0$; note that this operator is positive and rank one, so a non-zero scalar multiple of a projection, say $r$.  As $a$ normalises $B$, the element $r$ is in the cut-down $(1-p)B$, which is a commutative $C^*$-algebra as $p$ commutes with $B$.  Now, $r$ is in $A$ as it is rank one and $A$ contains the compacts.  Hence it is in $B$ as this $C^*$-algebra is maximal abelian in $A$ and as $r$ commutes with $B\subseteq pB\oplus (1-p)B$.  However, $r$ is orthogonal to $p$, a contradiction. Therefore $p$ commutes with $\mathcal N_A(B)$.

Finally, as $B\subseteq A$ is a Cartan subalgebra, $\mathcal{N}_A(B)$ generates $A$ as a $C^*$-algebra, and thus $p$ commutes with everything in $A$.  As $A$ contains the compacts, this forces $p=1$.
\end{proof}

Recall that if $S$ is a subset of $\mathcal{B}(H)$, then $C^*(S)$ denotes the $C^*$-algebra generated by $S$, and $W^*(S)$ the von Neumann algebra generated by $S$. 

\begin{theorem}\label{cart com the}
Let $A\subseteq \mathcal{B}(H)$ be a concrete $C^*$-algebra that contains the compact operators on $H$, and let $B\subseteq A$ be a Cartan subalgebra.  Then there exists a complete orthogonal set of rank one projections $\{p_i\}_{i\in I}$ on $H$ such that 
$$
C^*(\{p_i\}_{i\in I})\subseteq B \subseteq vN(\{p_i\}_{i\in I}).
$$
\end{theorem}
\begin{proof}
Let $\{p_i\}_{i\in I}$ be the complete set of orthogonal rank one projections in $B$ given by Lemma \ref{onb lem}.  As $B$ is a $C^*$-algebra, it contains $C^*(\{p_i\})$.  As $W^*(\{p_i\})$ is the maximal abelian $^*$-subalgebra of $\mathcal{B}(H)$ that contains $C^*(\{p_i\})$, $B$ is contained in $W^*(\{p_i\})$.
\end{proof}

Note that the conclusion of Theorem \ref{cart com the} on the structure of $B$ is best possible with those assumptions.  Indeed, if $\{p_i\}_{i\in I}$ is a complete orthogonal set of rank one projections on $H$, and $B$ is a $C^*$-subalgebra of $\mathcal{B}(H)$ with 
\begin{equation}
C^*(\{p_i\})\subseteq B\subseteq W^*(\{p_i\})
\end{equation}
then $A:=B+\mathcal{K}(H)$ clearly contains $B$ as a Cartan subalgebra.

On the other hand, we have the following observation giving some sufficient conditions for $B$ to equal $W^*(\{p_i\})$, which will play a role later in the paper.

\begin{proposition}\label{full basis}
Let $B\subseteq \mathcal{B}(H)$ be a concrete $C^*$-algebra such that there is a complete orthogonal set $\{p_i\}_{i\in I}$ of rank one projections such that 
\begin{equation}
C^*(\{p_i\})\subseteq B\subseteq W^*(\{p_i\}).
\end{equation}
Assume moreover that either:
\begin{enumerate}[(i)]
\itemsep-.1cm 
\item $B$ is closed in the strong topology\footnote{When $B$ is contained in a $C^*$-algebra $A\subseteq\mathcal B(H)$ containing the compact operators as in Theorem \ref{cart com the}, this can be defined in a representation independent way using that $b_n\to b$ strongly if and only if $b_nf \to bf$ in norm for each finite rank $f\in A$; this can be made sense of in a representation independent way as the finite rank operators are the unique minimal algebraic ideal of $A$.}; or
\item $B$ is abstractly $^*$-isomorphic to $\ell^\infty(X)$ for some set $X$.
\end{enumerate}
Then $B$ equals $W^*(\{p_i\})$.  
\end{proposition}

\begin{proof}
As the strong closure of $C^*(\{p_i\})$ equals $W^*(\{p_i\})$, part (i) is clear.  For part (ii), let $\phi:B\to \ell^\infty(X)$ be an abstract $^*$-isomorphism.  As $\phi$ must take the family $\{p_i\}_{i\in I}$ of minimal projections in $B$ bijectively to the family $\{q_x\}_{x\in X}$ of minimal projections in $\ell^\infty(X)$, it induces a bijection $f:I\to X$.  Note that if $S\subseteq I$ and $q_{f(S)}:=\sum_{i\in S}q_{f(i)}$ is the corresponding projection in $\ell^\infty(X)$, then $\phi^{-1}(q_{f(S)})$ is a projection on $H$ that commutes with the set $\{p_i\}_{i\in I}$, and that satisfies 
\begin{equation}
\phi^{-1}(q_{f(S)})p_i=\left\{\begin{array}{ll} p_i & i\in S \\ 0 & i\not\in S\end{array}\right..
\end{equation}
This is only possible if $\phi^{-1}(q_{f(S)})$ equals the projection $p_S:=\sum_{i\in S}p_i$ on $H$.  Hence $p_S$ is in $B$, and as $S$ was arbitrary, $B$ contains all projections in $W^*(\{p_i\})$.  The projections in $W^*(\{p_i\})$ span a norm-dense subset, however, so this gives us $B=W^*(\{p_i\})$.
\end{proof}

The next lemma adds another assumption on $A$ in order to limit the structure of $B$ a little more.  In order to state it, we introduce a little more notation. We will work on a separable Hilbert space, so any complete orthogonal set of projections can, and will, be indexed by $\mathbb N$. With this additional assumption we use the notation of Theorem \ref{cart com the}, and let $\{p_n\}_{n=1}^\infty$ be as in the conclusion, so in particular 
\begin{equation}
C^*(\{p_n\})\subseteq B\subseteq W^*(\{p_n\}).
\end{equation}
Assume moreover $A$ is unital, whence $B$ is too.  Then the spectrum of $B$ is a compact Hausdorff set $\widehat{B}$ that contains a copy of $\mathbb N$ as an open, dense, discrete subset; indeed, this follows as $C^*(\{p_n\})$ is an essential ideal in $B$, and the spectrum of $C^*(\{p_n\})$ identifies with $\mathbb N$.  Write $\widehat{B}_\infty:=\widehat{B}\setminus \mathbb N$, so that $\widehat{B}_\infty$ is a closed subset of $\widehat{B}$; by density of $\mathbb N$ in $\widehat{B}$, note that every point in $\widehat{B}_\infty$ is a limit of a net from $\N$.  Since a uniform Roe algebra satisfies the conditions on $A$ below, the following result also proves Proposition \ref{IntroA} from the introduction.

\begin{proposition}\label{non sep}
Let $A\subseteq \mathcal{B}(H)$ be a concrete unital $C^*$-algebra containing the compact operators, and assume that $H$ is infinite dimensional and separable.   Let $B\subseteq A$ be a Cartan subalgebra, with
\begin{equation}
C^*(\{p_n\})\subseteq B\subseteq W^*(\{p_n\}).
\end{equation}
as above.  Assume moreover that there is another complete orthogonal set of projections $\{q_n\}_{n=1}^\infty$ for $H$ such that $A$ contains $W^*(\{q_n\})$.  Then no element of $\widehat{B}_\infty$ is the limit of a sequence from $\mathbb N$. In particular $B$ is non-separable.
\end{proposition}

The assumptions of the lemma apply if $A$ is the uniform Roe algebra of a bounded geometry metric space (see Definition \ref{roe alg} below), and $B$ any Cartan subalgebra of $A$.  One can think of the lemma as saying that the topology of the spectrum of $B$ must be fairly complicated, and in particular $B$ cannot be separable. However, it does not imply that $B$ is all of $\ell^\infty$ as we will see in Example \ref{bad cartan ex} below.

\begin{proof}
We will identify $\widehat{B}=\N\sqcup \widehat{B}_\infty$, and write $\{p_n\}_{n\in \N}$ for the complete orthogonal set of projections that we started with; in terms of the spectrum $\widehat{B}=\N\sqcup \widehat{B}_\infty$ of $B$, $p_n$ can be thought of as the characteristic function of the singleton $\{n\}$.  For each $r\in \N$, let $Q_r\in W^*(\{q_n\})$ be defined by
\begin{equation}
Q_r:=q_1+\cdots q_r
\end{equation}
and set $Q_0=0$.

Assume for contradiction that there is some point $x_\infty\in \widehat{B}_\infty$ and a sequence in $\N$ that converges to it.  We will iteratively construct strictly increasing subsequences $(n_k)_{k=1}^\infty$ and $(m_k)_{k=1}^\infty$ of the given sequence converging to $x_\infty$, a strictly increasing sequence $(r_k)_{k=1}^\infty$ in $\N\cup\{0\}$, and a sequence $(e_k)_{k=1}^\infty$ of mutually orthogonal finite rank projections in $W^*(\{q_j\})$ with the following properties:
\begin{enumerate}[(i)]
\item $\|p_{n_k}e_kp_{n_k}\|>3/4$ for all $k$\label{p big};
\item $\|p_{m_k}e_jp_{m_k}\|<(1/4)2^{-j}$ for all $k$ and all $j\in \{1,...,k\}$; \label{p small}
\item $e_k\leq  1-Q_{r_k}$ for all $k$; \label{big e}
\item $\|p_{m_j}Q_{r_k}p_{m_j}\|>3/4$ for all $j\in \{1,...,k-1\}$. \label{big q}
\end{enumerate}
Indeed, to start the process off with $k=1$, set $r_1:=0$, so $Q_{r_1}=0$.  Let $n_1$ be the first element of the given sequence that converges to $x_\infty$, and choose $e_1=Q_r$ where $r$ is large enough that \eqref{p big} holds. Now choose $m_1$ large enough in the given sequence so that \eqref{p small} holds.  Note that  
\eqref{big q} and \eqref{big e} are vacuous.  Now, say we have constructed the desired elements up to stage $k$.  Choose $r_{k+1}>r_k$ large enough so that  
\eqref{big q} holds.  Choose $n_{k+1}>n_k$ far enough along the sequence converging to $x_\infty$ so that $\|p_{n_{k+1}}Q_{r_k}p_{n_{k+1}}\|<1/4$. Then choose $e_{k+1}$ so that \eqref{p big} and \eqref{big e} hold.  Finally,  choose $m_{k+1}>m_k$ far enough along the given sequence so that \eqref{p small} holds.  It is not too difficult to show that the resulting sequences have the claimed properties.

Now, given the above, set $e:=\sum_{k=1}^\infty e_k$, which converges strongly to an element of $W^*(\{q_j\})$.  Let $E:A\to B$ be the conditional expectation.  Thinking of elements of $B$ as functions on $\N$, we have that $E(e)$ is the function $f:n\mapsto \|p_nep_n\|$.  On the one hand, note that \eqref{p big} gives
\begin{equation}\label{e bigg}
\|p_{n_k}ep_{n_k}\|\geq \|p_{n_k}e_kp_{n_k}\|>3/4
\end{equation}
for each $k$. On the other hand, we have
\begin{eqnarray}\label{e smalll}
\|p_{m_k}ep_{m_k}\|&\leq& \sum_{j=1}^k \|p_{m_k}e_jp_{m_k}\|+\|p_{m_k}(\sum_{j=k+1}^\infty e_j)p_{m_k}\|\nonumber\\
&\stackrel{\eqref{p small},\ \eqref{big e}}{\leq}&\frac{1}{4}+\|p_{m_k}(1-Q_{r_{k+1}})p_{m_k}\|\\
&\stackrel{\eqref{big q}}{\leq}&\frac{1}{2}.
\end{eqnarray}
Now, as both sequences $(n_k)$ and $(m_k)$ converge to $x_\infty$, we have that  
\begin{equation}
f(x_\infty)=\lim_{k\to\infty}f(n_k)=\lim_{k\to\infty} \|p_{n_k}ep_{n_k}\|\geq 3/4
\end{equation}
from \eqref{e bigg}, and that 
\begin{equation}
f(x_\infty)=\lim_{k\to\infty}f(m_k)=\lim_{k\to\infty} \|p_{m_k}ep_{m_k}\|\leq 1/2
\end{equation}
from \eqref{e smalll}, giving us the desired contradiction.
\end{proof}

\section{An exotic Cartan subalgebra of a uniform Roe algebra}\label{S3}

In this short section we give an example of a Cartan subalgebra of a uniform Roe algebra with `exotic' spectrum.  We begin by recalling the definitions of bounded geometry metric spaces and the associated uniform Roe algebras.

\begin{definition}\label{bg def}
A metric space $X$ has \emph{bounded geometry} if for all $r>0$ there is $n_r\in \N$ such that all balls in $X$ of radius $r$ have at most $n_r$ elements. 

A function $f:X\to Y$ between metric spaces is \emph{uniformly expansive} if for all $r>0$ we have that 
$$
\sup_{x,y\in X,~d_X(x,y)\leq r} d_Y(f(x),f(y))<\infty.
$$ 
A function $f:X\to Y$ is a \emph{coarse equivalence} if it is uniformly expansive, and if there is a uniformly expansive function $g:Y\to X$ such that 
$$
\sup_{x\in X} d_X(x,g(f(x)))<\infty \quad \text{and}\quad\sup_{y\in Y} d_Y(y,f(g(y)))<\infty.
$$
Metric spaces $X$ and $Y$ are called \emph{coarsely equivalent} when there exists a coarse equivalence $f:X\rightarrow Y$.
\end{definition}

\begin{definition}\label{roe alg}
Let $X$ be a bounded geometry metric space, and let $a$ be a bounded operator on $\ell^2(X)$, which we think of as an $X$-by-$X$ matrix $a=(a_{xy})_{x,y\in X}$.  The \emph{propagation} of $a$ is
$$
\text{prop}(a):=\sup\{d(x,y)\mid a_{xy}\neq 0\}\in [0,\infty].
$$
Let $\C_u[X]$ denote the collection of bounded operators on $\ell^2(X)$ with finite propagation; this is a $^*$-algebra.  The \emph{uniform Roe algebra} of $X$, denoted $C^*_u(X)$, is the closure of $\C_u[X]$ for the operator norm.
\end{definition}

As a special case, note that if $X$ is a finitely generated group $\Gamma$ equipped with some word metric, then $C^*_u(X)$ is naturally $^*$-isomorphic to $\ell^\infty(\Gamma)\rtimes_r\Gamma$; this is proved for example in \cite[Proposition 5.1.3]{Brown:2008qy}.  

The uniform Roe algebra of a bounded geometry metric space always contains the compact operators $\mathcal{K}(\ell^2(X))$, as an essential ideal (note that $\mathcal{K}(\ell^2(X))$ is also the unique minimal $C^*$-ideal), and hence fits into the framework of the previous section.  Moreover, the subalgebra $\ell^\infty(X)$ of multiplication operators is a Cartan subalgebra (we prove this in more generality in Proposition \ref{l inf is cartan} below); hence in particular Proposition \ref{non sep} applies to uniform Roe algebras.

\begin{example}\label{bad cartan ex}
Let 
\begin{equation}
X:=\{n^2\mid n\in \N\}
\end{equation}
be the space of square numbers\footnote{There is nothing particularly special about the sequence $(n^2)$ here: any strictly increasing subsequence $(a_n)$ of $\N$ such that $|a_{n+1}-a_n|\to \infty$ as $n\to\infty$ would work just as well.} equipped with the metric it inherits as a subspace of $\N$. Note that we have 
\begin{equation}\label{roe alg 2}
C^*_u(X)=\ell^\infty(X)+\mathcal{K}(\ell^2(X)).
\end{equation}
This follows as the points of $X$ get more and more widely spaced, whence the only finite propagation operators are of the form `diagonal plus finite rank'.

Now, for each $n\in \N$, let $\xi_n:=\frac{1}{\sqrt{2}}(\delta_{(2n-1)^2}+\delta_{(2n)^2})$ and $\eta_n:=\frac{1}{\sqrt{2}}(\delta_{(2n-1)^2}-\delta_{(2n)^2})$, so the set 
\begin{equation}
S:=\{\xi_n,\eta_n\mid n\in \N\}
\end{equation}
is an orthonormal basis for $\ell^2(X)$.  Let $\ell^\infty(S)$ be the corresponding $C^*$-algebra of multiplication operators on $\ell^2(X)$, and define $B:=C^*_u(X)\cap \ell^\infty(S)$.  Thinking of $\ell^2(X)$ as decomposed into a direct sum of two dimensional subspaces
\begin{equation}
\ell^2(X)=\bigoplus_{n\geq 1} \ell^2(\{(2n-1)^2,(2n)^2\}),
\end{equation}
operators in $\ell^\infty(S)$ look like 
\begin{equation}\label{mat prod}
\prod_{n\geq 1} \begin{pmatrix} a_n & b_n \\ b_n & a_n \end{pmatrix},
\end{equation}
where $(a_n)$ and $(b_n)$ are arbitrary bounded sequences.  Elements of $B$ look like this, except now we must also ask that $b_n\to 0$ as $n\to\infty$ (it is straightforward to check that this is a necessary and sufficient for such an operator from $\ell^\infty(S)$ to be in $C^*_u(X)$). 

We claim the algebra $B$ is a Cartan subalgebra of $C^*_u(X)$.  This follows from the computations below.
\begin{enumerate}[(i)]
\item It is maximal abelian:  The algebra $B$ contains $C_0(S)$.  The commutant of $C_0(S)$ in $\mathcal{B}(\ell^2(X))$ is $\ell^\infty(S)$, and thus $B$ contains everything in $C^*_u(X)$ that commutes with $C_0(S)$, and in particular contains everything that commutes with $B$ itself.  
\item The normaliser $\mathcal{N}_{C^*_u(X)}(B)$ generates $C^*_u(X)$:  Indeed, thinking of operators in $B$ as matrices as in \eqref{mat prod} above, we see that the normaliser of $B$ in $C^*_u(X)$ contains all products of matrices of the form 
\begin{equation}
\prod_{n\geq 1} \begin{pmatrix} c_n & 0 \\ 0 & c_n \end{pmatrix}\quad\text{and}\quad\prod_{n\geq 1} \begin{pmatrix} d_n & 0 \\ 0 & -d_n \end{pmatrix},
\end{equation}
where $(c_n)$, $(d_n)$ are arbitrary bounded sequences.  Clearly then the $C^*$-algebra generated by the normaliser $\mathcal{N}_{C^*_u(X)}(B)$ contains $\ell^\infty(X)$.  It also straightforward to see that it contains $\mathcal{K}(\ell^2(X))$, and so by \eqref{roe alg 2} is all of $C^*_u(X)$.
\item There is a faithful conditional expectation $C^*_u(X)\to B$.  Let $E:\mathcal{B}(\ell^2(X))\to \ell^\infty(S)$ be the canonical conditional expectation, which is faithful.   We need to check that $E$ takes $C^*_u(X)$ onto $B$ (and not onto some larger subalgebra of $\ell^\infty(S)$).  Looking at line \eqref{roe alg} above, $E$ takes $\mathcal{K}(\ell^2(X))$ to $C_0(S)\subseteq B$, so it suffices to check that $E(\ell^\infty(X))\subseteq B$.  With respect to a matrix decomposition as in \eqref{mat prod} above, an arbitrary element of $\ell^\infty(X)$ looks like 
\begin{equation}
\prod_{n\geq 1} \begin{pmatrix} a_n & 0 \\ 0 & b_n \end{pmatrix}
\end{equation}
for some bounded sequences $(a_n)$ and $(b_n)$.  The computation of the image of this element under $E$ may be performed one matrix at a time.   Doing this, with $E_n$ the restriction of $E$ to the bounded operators on $\ell^2(\{(2n-1)^2,(2n)^2\})$, we see that 
\begin{align}
E_n\Big(\begin{pmatrix} a_n & 0 \\ 0 & b_n \end{pmatrix}\Big) & =\frac{1}{2}\begin{pmatrix} 1 & 1 \\ 1 & 1 \end{pmatrix}\begin{pmatrix} a_n & 0 \\ 0 & b_n \end{pmatrix}\frac{1}{2}\begin{pmatrix} 1 & 1 \\ 1 & 1 \end{pmatrix}+\nonumber \\ & \quad \quad \quad \frac{1}{2}\begin{pmatrix} 1 & -1 \\ -1 & 1 \end{pmatrix}\begin{pmatrix} a_n & 0 \\ 0 & b_n\end{pmatrix} \frac{1}{2}\begin{pmatrix} 1 & -1 \\ -1 & 1 \end{pmatrix} \nonumber\\ 
& =\frac{1}{2}\begin{pmatrix} a_n+b_n & 0 \\ 0 & a_n+b_n \end{pmatrix}
\end{align}
and this is certainly in $B$.
\end{enumerate}
\end{example}

\begin{remark}\label{bad cart cosep}
The Cartan subalgebra $B$ above is co-separable in the sense of Definition \ref{cosepdef}, and indeed we do not know if it is possible for the uniform Roe algebra of a bounded geometry metric space to admit a Cartan subalgebra that is not co-separable.  To see co-separability of $B$, let $S_0$ be a countable subset of $\mathcal{N}_{C^*_u(X)}(B)$ that generates $\mathcal{K}(\ell^2(X))$, and with our usual matrix conventions, let $s$ be the element 
\begin{equation}
s:=\prod_{n\geq 1} \begin{pmatrix} 1 & 0 \\ 0 & -1\end{pmatrix}
\end{equation}
of $\ell^\infty(X)$, which normalizes $B$.  Set $S:=S_0\cup \{s\}$.   We claim that $S$ and $B$ together generate $C^*_u(X)$.  By assumption on $S_0$ and line \eqref{roe alg 2}, it suffices to show that the $C^*$-algebra generated by $s$ and $B$ contains $\ell^\infty(X)$.  Let then 
\begin{equation}
\prod_{n\geq 1}  \begin{pmatrix} a_n & 0 \\ 0 & b_n\end{pmatrix}
\end{equation}
be an arbitrary element of $\ell^\infty(X)$, and note that 
\begin{align}
\nonumber\prod_{n\in \N}  \begin{pmatrix} a_n & 0 \\ 0 & b_n\end{pmatrix} & =\frac{1}{2}\prod_{n\in \N}  \begin{pmatrix} a_n+b_n & 0 \\ 0 & a_n+b_n\end{pmatrix} \\ & \quad \quad \quad+s\frac{1}{2}\prod_{n\in \N}  \begin{pmatrix} a_n-b_n & 0 \\ 0 & a_n-b_n\end{pmatrix}~;
\end{align}
as the two products of matrices on the right hand side are in $B$, we are done.
\end{remark}

\begin{remark}\label{not diag}
Recall from \cite{Kumjian:1986aa} and \cite[Page 55]{Renault:2008if} that a Cartan subalgebra $B\subseteq A$ in a $C^*$-algebra is a \emph{$C^*$-diagonal} if every pure state on $B$ extends uniquely to a (necessarily pure) state on $A$.  The usual Cartan subalgebra $\ell^\infty(X)$ in a uniform Roe algebra $C^*_u(X)$ is a $C^*$-diagonal, as is not difficult to check directly (this also follows from \cite[Proposition 5.11]{Renault:2008if}, and the fact that the underlying coarse groupoid is principal).  The exotic Cartan subalgebra of Example \ref{bad cartan ex} is not a $C^*$-diagonal, however.  To see this, fix a non-principal ultrafilter $\omega$ on $\N$, and note that the state on $B$ defined on matrices as in line \eqref{mat prod} above by 
\begin{equation}
\prod_{n\geq 1} \begin{pmatrix} a_n & b_n \\ b_n & a_n \end{pmatrix}\mapsto \lim_{n\to\omega} a_n
\end{equation}
is pure: indeed, the fact that the sequence $(b_n)$ of off-diagonal entries is in $C_0(\N)$ implies that it is a $^*$-homomorphism.  However, it admits two different pure extensions to $C^*_u(X)$: indeed, if $a\in C^*_u(X)$ has diagonal entries given by $a_{m^2~m^2}$, these can be defined by
\begin{equation}
a\mapsto \lim_{n\to\omega} a_{(2n-1)^2~(2n-1)^2} \quad \text{and}\quad a\mapsto \lim_{n\to\omega} a_{(2n)^2~(2n)^2}.
\end{equation}
We do not know if there exist uniform Roe algebras that admit exotic $C^*$-diagonals.
\end{remark}

\section{Abstract coarse structures and Roe Cartan subalgebras}\label{abs cart sec}

Our goal in this section is to prove that $C^*$-algebras containing the compact operators as an essential ideal and a co-separable Cartan masa which is abstractly isomorphic to $\ell^\infty(\mathbb N)$ are essentially the same thing as bounded geometry metric spaces (considered up to bijective coarse equivalence).  Actually, we work in more generality than this, using the language of abstract coarse structures as this seems to give slightly cleaner results.

The following definition is due to Roe \cite[Chapter 2]{Roe:2003rw}.

\begin{definition}\label{cs}
Let $X$ be a set.  A \emph{coarse structure on $X$} is a collection $\mathcal{E}$ of subsets of $X\times X$ such that:
\begin{enumerate}[(i)]
\itemsep-.1cm 
\item for all $E,F\in \mathcal{E}$ the union $E\cup F$ is in $\mathcal{E}$;
\item for all $E,F\in \mathcal{E}$, the composition 
\begin{equation}
E\circ F:=\{(x,z)\in X\times X\mid \exists y\in X \text{ with } (x,y)\in E \text{ and } (y,z)\in F\}
\end{equation}
is in $\mathcal{E}$;
\item for all $E\in \mathcal{E}$, the inverse 
\begin{equation}
E^{-1}:=\{(x,y)\in X\times X\mid (y,x)\in E\}
\end{equation}
is in $\mathcal{E}$;
\item for all $E\in \mathcal{E}$, if $F\subseteq E$, then $F\in \mathcal{E}$;
\item \label{unital cs} $\mathcal{E}$ contains the diagonal $\{(x,x)\in X\times X\mid x\in X\}$.
\end{enumerate}
A set $X$ together with a coarse structure $\mathcal E$ is called a \emph{coarse space}; when it is unlikely to cause confusion, we will leave $\mathcal{E}$ implicit, and just say that $X$ is a coarse space.

A coarse space $(X,\mathcal{E})$ is:
\begin{enumerate}[(a)]
\itemsep-.1cm 
\item \label{bg part} of \emph{bounded geometry} if for all $E\in \mathcal{E}$, the cardinalities of the `slices'
\begin{equation}
E_x:=\{(y,x)\in E\mid y\in X\} \text{ and } E^x:=\{(x,y)\in E\mid y\in X\}
\end{equation}
are bounded independently of $x$;
\item \emph{connected} if for every $x,y\in X$, $\mathcal{E}$ contains $\{(x,y)\}$;
\item \emph{countably generated} if there is a countable collection $S$ of subsets of $X\times X$ such that $\mathcal{E}$ is generated by $S$ (i.e.\ such that $\mathcal{E}$ is the intersection of all coarse structures containing $S$). 
\end{enumerate}
\end{definition}

The basic example of a coarse structure is the \emph{bounded coarse structure} on a metric space $(X,d)$, defined by 
\begin{equation}
\mathcal{E}_d:=\{E\subseteq X\times X\mid d|_{E} \text{ is bounded}\}
\end{equation}
(it is straightforward to check that this is a coarse structure).  A coarse space $(X,\mathcal{E})$ is \emph{metrisable} if there exists a metric $d$ on $X$ such that $\mathcal{E}$ is the associated bounded coarse structure.   Note that the bounded coarse structure associated to a metric has bounded geometry if and only if the metric does in the usual sense of Definition \ref{bg def} above. The bounded coarse structure is connected and generated by the countably many sets 
\begin{equation}
E_n:=\{(x,y)\in X\times X\mid d(x,y)\leq n\},
\end{equation}
for $n\in\mathbb N$, $q\geq 0$. Conversely, one has the following result: see \cite[Theorem 2.55]{Roe:2003rw} for a proof.

\begin{theorem}\label{met the}
A coarse space $X$ is metrizable if and only if it is connected and countably generated. \qed
\end{theorem}

The following combinatorial lemma (a standard `greedy algorithm' argument) will be used several times below.

\begin{lemma}\label{greedy}
Let $(X,\mathcal{E})$ be a bounded geometry coarse space and $E$ be an element of $\mathcal{E}$.  Then there exists $N\in \N$ and a decomposition 
\begin{equation}
E=\bigsqcup_{n=1}^N E_n
\end{equation}
of $E$ into disjoint subsets such that for each $x\in X$ and each $n$, there is at most one element in each set 
\begin{equation}
E_n\cap\{(x,y)\mid y\in X\} \quad \text{and} \quad E_n\cap \{(y,x)\mid y\in X\}
\end{equation}
(in words, $E_n$ intersects each row and column in $X\times X$ at most once).
\end{lemma}

\begin{proof}
Set $E_0$ to be the empty set.  Having chosen disjoint subsets $E_0,E_1,...,E_n$ of $E$, set $E_{n+1}$ to be a maximal subset of $E\setminus (E_1\cup \cdots \cup E_n)$ that intersects each row and column at most once.  We claim that for some $N$, $E_n$ is empty for all $n\geq N$.  Indeed, if not, then for every $N$, there is some element $(x_N,y_N)$ in $E_N$, and in particular that has not appeared in any of $E_1,...,E_{N-1}$.  Maximality of these sets implies that for each $n\in \{1,...,N-1\}$ there is either $x_n$ such that $(x_n,y_N)$ is in $E_n$, or $y_n$ such that $(x_N,y_n)$ is in $E_n$.  This implies that at least one of the sets 
\begin{equation}
\{(x_N,y)\in E\mid y\in X\} \quad \text{or}\quad \{(x,y_N)\in E\mid x\in X\}
\end{equation}
has cardinality at least $\lfloor (N-1)/2\rfloor$.  As this happens for all $N$, this contradicts that $(X,\mathcal{E})$ has bounded geometry.
\end{proof}

We now turn to bounded operators.  We start with a basic class of operators.  

\begin{definition}\label{sd}
Let $\{\xi_i\}_{i\in I}$ be an orthonormal basis for a Hilbert space $H$.  For any bounded operator $a$ on $H$, let $a_{ij}:=\langle \xi_i,a\xi_j\rangle$ be the corresponding matrix entries.  We will say that a matrix $(a_{ij})$, or the operator defining it (if one exists) is \emph{supported on a single diagonal (with respect to $\{\xi_i\}_{i\in I}$)} if for each $i$ there is at most one $j$ such that $a_{ij}\neq 0$, and at most one $k$ such that $a_{ki}\neq 0$ (in words, $a$ has at most one non-zero matrix entry in each row and column).
\end{definition}

The following elementary lemma is well-known.

\begin{lemma}\label{linf normal}
Let $\{\xi_i\}_{i\in I}$ be an orthonormal basis for $H$, and let $\{p_i\}_{i\in I}$ be the corresponding complete set of orthogonal rank one projections.  
\begin{enumerate}[(i)]
\item Let $B\subseteq \mathcal{B}(H)$ be a $C^*$-algebra such that 
\begin{equation}
C^*(\{p_i\})\subseteq B\subseteq W^*(\{p_i\}).
\end{equation}
Then if $a\in \mathcal{B}(H)$ normalises $B$, we have that $a$ is supported on a single diagonal with respect to the basis $\{\xi_i\}_{i\in I}$.  If $B=W^*(\{p_i\})$ and $a$ is supported on a single diagonal, then $a$ normalises $B$.
\item Let $(a_{ij})_{i,j\in I}$ be a matrix supported on a single diagonal (not necessarily coming from a bounded operator).  Then matrix multiplication by $(a_{ij})$ defines a bounded operator $a$ if and only if its matrix entries are uniformly bounded, and in this case, $\|a\|=\sup_{i,j}|a_{ij}|$.   
\end{enumerate}
\end{lemma}

\begin{proof}
For each $i\in I$ and $a\in \mathcal{B}(H)$, the operators $ap_ia^*$ and $a^*p_ia$ have matrix entries given by 
\begin{equation}
(ap_ia^*)_{jk}=a_{ji}\overline{a_{ki}} \quad \text{and} \quad (a^*p_ia)_{jk}=\overline{a_{ij}}a_{ik}
\end{equation}
respectively.  As $B\subseteq W^*(\{p_i\})$, in order for these operators to be in $B$ for some fixed $i$ the entries can only be non-zero if $j=k$, which can only happen if $a$ has at most one non-zero entry in each row and column.  

For the second part of the statement, note that the computation above shows that if $a$ is supported on a single diagonal, then $ap_ia^*$ and $a^*p_ia$ are in $B$ for each $i$.  As we are now assuming that $B$ is generated by the $p_i$ (as a von Neumann algebra), this completes the proof.

For part (ii), assume $a$ is supported on a single diagonal.  Note that $a\xi_i=a_{t(i)i}\xi_{t(i)}$, where $t(i)$ is the unique element of $I$ such that $a_{t(i)i}\neq 0$, or $a\xi_i=0$ if no such $t(i)$ exists.  Moreover, if $i\neq j$, then $a\xi_i$ is orthogonal to $a\xi_j$.  Hence for any element $v:=\sum_{i\in I}\lambda_i \xi_i$ of $H$,  
\begin{align}
\nonumber\|av\|^2 & =\sum_{i\in I,t(i) \text{ exists}}\|a_{t(i)i}\lambda_iv_{t(i)}\|^2\leq \sup_{i\in I,t(i)\text{ exists}}|a_{it(i)}|^2\sum_{i\in I}|\lambda_i|^2 \\ & =\sup_{i,j}|a_{ij}|^2\|v\|^2.
\end{align}
This gives $\|a\|\leq \sup_{i,j}|a_{ij}|$; the opposite inequality follows as $\|a\|\geq |\langle \xi_i,a\xi_j\rangle|$ for any $i,j$.
\end{proof}

As a special case of Definition \ref{sd}, we equip $\ell^2(X)$ with its canonical orthonormal basis $\{\delta_x\}_{x\in X}$, so the matrix entries of a bounded operator $a$ on $\ell^2(X)$ are $a_{xy}:=\langle \delta_x,a\delta_y\rangle$.  It is routine to extend the definition of uniform Roe algebras from metric spaces to general coarse spaces of bounded geometry; we set out the details below.

\begin{definition}\label{fin prop}
Let $(X,\mathcal{E})$ be a coarse space.  The \emph{support} of an operator $a\in \mathcal{B}(\ell^2(X))$ is 
\begin{equation}
\text{supp}(a):=\{(x,y)\in X\times X\mid a_{xy}\neq 0\}.
\end{equation}
An operator $a\in \mathcal{B}(\ell^2(X))$ has \emph{finite propagation} if its support is an element of $\mathcal{E}$.
\end{definition}
From the axioms for a coarse structure, it is not difficult to check that the collection of finite propagation operators is a $^*$-algebra, leading to the following definition.

\begin{definition}\label{u roe}
Let $(X,\mathcal{E})$ be a coarse space.  Let  $\C_u[X;\mathcal{E}]$ denote the $^*$-algebra of finite propagation operators on $\mathcal{B}(\ell^2(X))$ as in Definition \ref{fin prop}.  The \emph{uniform Roe algebra of $X$}, denoted $C^*_u(X;\mathcal{E})$, is the norm closure of $\C_u[X;\mathcal{E}]$.
\end{definition}

Note that if a coarse space $(X,\mathcal{E})$ is metrisable, then $C^*_u(X;\mathcal{E})$ is the usual uniform Roe algebra associated to any choice of metric $d$ on $X$ which induces $\mathcal E$.  Note also that condition \eqref{unital cs} in Definition \ref{cs} above implies that the multiplication operators $\ell^\infty(X)\subseteq \mathcal{B}(\ell^2(X))$ are contained in $\C_u[X;\mathcal{E}]$.  

The following special class of operators in $\C_u[X;\mathcal{E}]$ will be useful for us.

\begin{definition}\label{p isoms}
Let $(X,\mathcal{E})$ be a bounded geometry coarse space.  For any $E\in \mathcal{E}$ with at most one entry in each row and column, define a matrix $(v^E_{xy})$ by the formula 
\begin{equation}
v^E_{xy}:=\left\{\begin{array}{ll} 1 & (x,y)\in E \\ 0 & (x,y)\not\in E\end{array}\right..
\end{equation}
Let $v^E$ denote the unique bounded operator on $\ell^2(X)$ associated to this matrix by  Lemma \ref{linf normal} part (ii).
\end{definition}

We now have the following useful structure lemma for $C^*_u(X;\mathcal{E})$ that holds whenever $X$ has bounded geometry.

\begin{lemma}\label{u roe struc}
With notation as in Definition \ref{p isoms}, $v^E$ is a partial isometry in $\C_u[X;\mathcal{E}]$ that normalizes $\ell^\infty(X)$.  Moreover, if $S\subseteq \mathcal{E}$ is a collection of subsets of $X\times X$, each with at most one entry in each row and column, and that generates the coarse structure, then the collection 
\begin{equation}
\{v^E\mid E\in S\}\cup \ell^\infty(X)
\end{equation}
generates $\C_u[X;\mathcal{E}]$ as a $^*$-algebra (and therefore generates $C^*_u(X;\mathcal{E})$ as a $C^*$-algebra).
\end{lemma}

\begin{proof}
That each $v^E$ is a partial isometry in $\C_u[X;\mathcal{E}]$ follows from straightforward computations, and each normalises $\ell^\infty(X)$ by Lemma \ref{linf normal} part (i).  

Let $S$ be a collection of sets as in the statement.  We consider the following collections of subsets of $X\times X$: $S_1$ consists of all compositions of the form $E_1\circ \cdots \circ E_n$ where for each $i$, either $E_i$ or its inverse is in $S$; $S_2$ consists of all subsets of elements of $S_1$; $S_3$ consists of all finite disjoint unions of sets from $S_2$.  As $S$ generates $\mathcal{E}$, it is not too difficult to see that in fact $S_3=\mathcal{E}$.  

Consider now the $^*$-subalgebra $\mathcal{A}$ of $\C_u[X;\mathcal{E}]$ generated by $\{v^E\mid E\in S\}\cup \ell^\infty(X)$.  We first claim that if the support of $a\in \C_u[X;\mathcal{E}]$ is contained in $S_1$, then $a$ is in $\mathcal{A}$.  Indeed, if $E$, $F$ are elements of $S$, then $v^{E^{-1}}=(v^E)^*$ and $v^{E\circ F}=v^Ev^F$, from which it follows that if $E\in S_1$, then $v^E$ is in $\mathcal{A}$.  Any $a\in \C[X;\mathcal{E}]$ with support an element of $S_1$ can be written as $fv^E$ for some $f\in \ell^\infty(X)$ and $E\in S_1$: indeed, take $E$ to be the support of $a$, and define $f$ by 
\begin{equation}
f(x):=\left\{\begin{array}{ll} a_{xy} & \exists y\in X \text{ with } (x,y)\in E \\ 0 & \text{otherwise}\end{array}\right. .
\end{equation}
This completes the proof of the claim.

We next claim that if $a\in \C_u[X;\mathcal{E}]$ has support in $S_2$, then $a$ is in $\mathcal{A}$.  Indeed, let $F$ be an element of $S_1$ containing the support $E$ of $a$, and define $b\in \C_u[X;\mathcal{E}]$ to have the matrix that agrees with that of $a$ on $(X\times X)\setminus F$ and on $\text{supp}(a)$, and has all entries in $F\setminus \text{supp}(a)$ equal to one.  Then the support of $b$ is in $F$, which is in $S_1$, and so $b$ is in $\mathcal{A}$ by our earlier claim.  Moreover, $a=fb$, where $f\in \ell^\infty(X)$ is the characteristic function of the set $\{x\in X\mid (x,y)\in E \text{ for some } y \}$, so we are done with this claim.

Finally, we claim that if $a\in \C_u[X;\mathcal{E}]$ has support in $S_3$, then $a$ is in $\mathcal{A}$; as $S_3=\mathcal{E}$, this will suffice to complete the proof.  For this, write the support $E$ of $a$ as a finite disjoint union $E=\bigsqcup_{i=1}^n E_i$ with each $E_i$ in $S_2$.  Then $a=\sum_{i=1}^n b_i$, where the matrix of each $b_i$ is defined to agree with the matrix of $a$ on $E_i$, and to be zero on $(X\times X)\setminus E_i$.  In particular $b_i$ is in $\mathcal{A}$ by the previous claim, and we are done.
\end{proof}

Just as in the case of metric spaces, the canonical copy of $\ell^\infty(X)$ forms a Cartan masa in the uniform Roe algebra.

\begin{proposition}\label{l inf is cartan}
Let $(X,\mathcal{E})$ be a bounded geometry coarse structure.  Then $\ell^\infty(X)$ is a Cartan subalgebra in $C^*_u(X;\mathcal{E})$.  Moreover, if $X$ is connected, then $C^*_u(X;\mathcal{E})$ contains the compact operators.
\end{proposition}

\begin{proof}
It is well-known that $\ell^\infty(X)$ is a unital, maximal abelian $C^*$-subalgebra of $\mathcal{B}(\ell^2(X))$ that is the image of a faithful conditional expectation $\mathcal{B}(\ell^2(X))\to \ell^\infty(X)$, so it certainly also has these properties when considered as a $C^*$-subalgebra of $C^*_u(X;\mathcal{E})$.  The normaliser of $\ell^\infty(X)$ generates $C^*_u(X;\mathcal{E})$ by Lemma \ref{u roe struc}, completing the proof that $\ell^\infty(X)$ is a Cartan subalgebra.  Assuming that $X$ is connected, then with notation as in Definition \ref{p isoms} we get that for any $(x,y)\in X\times X$ the operator $v^{\{(x,y)\}}$ is in $\C_u[X;\mathcal{E}]$.  These operators generate the compact operators, so we are done.
\end{proof}

To summarise, given a connected coarse space $(X,\mathcal E)$ of bounded geometry, $\ell^\infty(X)\subseteq C^*_u(X,\mathcal E)$ is a Cartan pair with the compact operators contained in $C^*_u(X,\mathcal E)$.  Our main goal of this section is to prove a sort of converse.

\begin{definition}\label{cart coarse def}
Let $A$ be a unital $C^*$-algebra containing a copy $\mathcal{K}$ of the compact operators on some Hilbert space as an essential ideal, and let $B\subseteq A$ be a Cartan subalgebra.  Let $X=\{p_x\}_{x\in X}$ be the set\footnote{We treat $X$ as its own index set; apologies for this abuse of notation.  It is non-empty, as we will see in Remark \ref{rep rem}.} of minimal projections in $B$, and for each $a\in \mathcal{N}_A(B)$ and each $\epsilon>0$, define 
\begin{equation}
E_{a,\epsilon}:=\{(x,y)\in X\times X\mid \|p_xap_y\|\geq \epsilon\}.
\end{equation}
Define $\mathcal{E}_A$ to be the coarse structure on $X$ generated by the collection 
\begin{equation}
\{E_{a,\epsilon}\mid a\in \mathcal{N}_A(B)\text{ and }\epsilon>0\}. 
\end{equation}
\end{definition}

\begin{remark}\label{rep rem}
With notation as in Definition \ref{cart coarse def}, fix a faithful irreducible representation of $\mathcal{K}$ on some Hilbert space $H$; such a representation exists, is unique up to unitary equivalence, and necessarily consists of an isomorphism $\mathcal{K}\cong \mathcal{K}(H)$ of $\mathcal{K}$ with the compact operators on $H$ (see for example \cite[Section 4.1]{Dixmier:1977vl}).  As $\mathcal{K}$ is an essential ideal in $A$, this representation extends uniquely to a representation of $A$ on $H$, which is also irreducible and faithful.  Identify $A$ with its image under this representation.  We may now apply Theorem \ref{cart com the}: this implies in particular that the set $\{p_x\}_{x\in X}$ of minimal projections in $B$ identifies with a complete collection of orthogonal rank one projections on $H$, and that 
\begin{equation}\label{cart incl}
C^*(\{p_x\})\subseteq B\subseteq W^*(\{p_x\}).
\end{equation}
\end{remark}

\begin{lemma}\label{cart coarse lem}
With notation as in Definition \ref{cart coarse def}, the coarse space $(X,\mathcal{E}_A)$ is connected and has bounded geometry.
\end{lemma}

\begin{proof}
Fix a representation $H$ of $A$ with the properties in Remark \ref{rep rem}.  For each $x$, choose a unit vector $\xi_x$ in the range of $p_x$, so the collection $\{\xi_x\}_{x\in X}$ is an orthonormal basis for $H$.  Use this basis to write operators on $H$ as matrices $(a_{xy})_{x,y\in X}$ as in Definition \ref{sd}.  Note that $|a_{xy}|=\|p_xap_y\|$ for any $x,y\in X$.

Now, as $A$ contains the compact operators, for any $(x,y)\in X\times X$, the operator $v^{\{(x,y)\}}$ whose matrix has a single entry equal to one in the $(x,y)^{\text{th}}$ position and zeros elsewhere is in $A$, and is moreover in $\mathcal{N}_A(B)$ by a direct calculation (just as in Lemma \ref{linf normal} part (i)); this implies that $\{(x,y)\}$ is in $\mathcal{E}_A$, and thus the coarse space $X$ is connected.

Let $S$ be the collection of all elements $E$ of $\mathcal{E}_A$ such that $E$ has at most one element in each row and column.  Then Lemma \ref{linf normal} part (ii) implies that each $E_{a,\epsilon}$ is in $S$ as $a$ ranges over $\mathcal{N}_A(B)$ and $\epsilon$ over $(0,\infty)$, whence $S$ generates $\mathcal{E}_A$.  Note that $S$ is closed under all the operations defining a coarse structure, except (possibly) unions.  It follows that $\mathcal{E}_A$ consists precisely of finite unions of sets from $S$, and thus has bounded geometry.
\end{proof}

\begin{lemma}\label{contained in roe}
With notation as in Definition \ref{cart coarse def}, identify $A$ with its image in some representation on a Hilbert space $H$ with the properties in Remark \ref{rep rem}.  For each $x\in X$, choose a unit vector $\xi_x$ in the range of $p_x$, so $\{\xi_x\}_{x\in X}$ is an orthonormal basis of $H$, and define a unitary isomorphism
\begin{equation}
u: \ell^2(X)\to H,\quad \delta_x\mapsto \xi_x.
\end{equation}
Consider $C^*_u(X;\mathcal{E}_A)$ and its Cartan subalgebra $\ell^\infty(X)$ as represented on $\ell^2(X)$ in the canonical way.  Then $u^*Bu$ is contained in $\ell^\infty(X)$, and $u^*Au$ is contained in $C^*_u(X;\mathcal{E}_A)$.
\end{lemma}

\begin{proof}
Note that $u^*p_xu$ is the orthogonal projection onto the span of $\delta_x$, whence $u^*(W^*(\{p_x\}))u=\ell^\infty(X)$.  Hence by line \eqref{cart incl} above, $uBu^*\subseteq \ell^\infty(X)$.  To see that $uAu^*\subseteq C^*_u(X;\mathcal{E}_A)$, it suffices to show that $u^*\mathcal{N}_A(B)u$ is contained in $C^*_u(X;\mathcal{E}_A)$.  Let then $a$ be an element of $\mathcal{N}_A(B)$ and let $\epsilon>0$. Then as the matrix associated to $a$ has at most one non-zero entry in each row and column by Lemma \ref{linf normal} part (i), part (ii) of that lemma implies that the operator $a^{(\epsilon)}$ with matrix entries
\begin{equation}
a^{(\epsilon)}_{xy}:=\left\{\begin{array}{ll} a_{xy} & |a_{xy}|\geq \epsilon \\ 0 & |a_{xy}|<\epsilon \end{array}\right.
\end{equation}
is well-defined, bounded, and that the collection $(a^{(\epsilon)})_{\epsilon>0}$ satisfies $\|a^{(\epsilon)}-a\|\to 0$ as $\epsilon\to 0$.  Clearly each conjugate $u^*a^{(\epsilon)}u$ is in $C^*_u(X;\mathcal{E}_A)$, however, so we are done.
\end{proof}

If $B$ is abstractly isomorphic to some $\ell^\infty(I)$ for some set $I$, then we can do better. In this case $B\subseteq A$ is unitarily equivalent to $\ell^\infty(X)\subseteq C^*_u(X,\mathcal E_A)$.

\begin{proposition}\label{equals roe}
With notation as in Lemma \ref{contained in roe}, assume moreover that $B$ is abstractly isomorphic to $\ell^\infty(I)$ for some set $I$.  Then the inclusions $u^*Bu\subseteq \ell^\infty(X)$, and $u^*Au\subseteq C^*_u(X;\mathcal{E}_A)$ are equalities. 
\end{proposition}

\begin{proof}
The fact that $u^*Bu=\ell^\infty(X)$ follows from part (ii) of Proposition \ref{full basis}.  To see that $u^*Au=C^*_u(X;\mathcal{E}_A)$, Lemma \ref{u roe struc} implies that it suffices to show that for each $a\in \mathcal{N}_A(B)$ and each $\epsilon>0$, if $E=E_{a,\epsilon}$, then the partial isometry $v^{E}$ is in $u^*Au$.  Define $f\in \ell^\infty(X)$ by 
\begin{equation}
f(x):=\left\{\begin{array}{ll} (a_{xy})^{-1} &\exists y\in X\text{ such that } |a_{xy}|\geq \epsilon \\ 0 &\text{otherwise}\end{array}\right.;
\end{equation}
this definition makes sense as Lemma \ref{linf normal} part (i) implies that the matrix underlying $a$ has at most one non-zero entry in each row.  Noting that $f\in \ell^\infty(X)=u^*Bu\subseteq u^*Au$, we get that $v^E=fu^*au$ is in $u^*Au$ and so the proof is complete.
\end{proof}

The next definition and theorem formalise much of the above discussion.

\begin{definition}\label{const}
Let $\mathcal{A}$ be the collection of triples $(A,B,\mathcal{K})$, where $A$ is a unital $C^*$-algebra, $B\subseteq A$ is a Cartan subalgebra abstractly $^*$-isomorphic to $\ell^\infty(I)$ for some set $I$, and $\mathcal{K}$ is an essential ideal of $A$ that is abstractly $^*$-isomorphic to the compact operators on some Hilbert space.   Let $\mathcal{X}$ be the collection of connected, bounded geometry coarse spaces $(X,\mathcal{E})$.

Define correspondences 
\begin{equation}
\Phi:\mathcal{X}\to \mathcal{A}, \quad  (X,\mathcal{E}) \mapsto (C^*_u(X;\mathcal{E}),\ell^\infty(X),\mathcal{K}(\ell^2(X)))
\end{equation}
(notation on the right as in Definition \ref{u roe}) and 
\begin{equation}
\Psi: \mathcal{A}\to \mathcal{X}, \quad (A,B,\mathcal{K})\mapsto (X,\mathcal{E}_A)
\end{equation}
(notation on the right as in Definition \ref{cart coarse def}).
\end{definition}

\begin{theorem}\label{cart coarse prop}
The two correspondences $\Phi:\mathcal{X}\to \mathcal{A}$ and $\Psi: \mathcal{A}\to \mathcal{X}$ are well-defined.  Moreover, the compositions $\Phi\circ \Psi$ and $\Psi\circ \Phi$ are both `isomorphic to the identity' in the following precise senses.

For $\Phi\circ \Psi$: for any triple $(A,B,\mathcal{K})\in\mathcal A$, let $H$ be a representation as in Remark \ref{rep rem}; then there is a unitary isomorphism $u:\ell^2(X)\to H$ such that 
\begin{equation}
u^*Au=C^*_u(X;\mathcal{E}_A), \quad u^*Bu=\ell^\infty(X),\quad \text{and}\quad u^*\mathcal{K}u=\mathcal{K}(\ell^2(X)).
\end{equation}
For $\Psi\circ \Phi$: for any $(X,\mathcal{E})\in\mathcal X$, letting $A=C^*_u(X;\mathcal{E})$ and identifying the set of minimal projections in $\ell^\infty(X)$ with $X$, we have that $\mathcal{E}=\mathcal{E}_A$.
\end{theorem}

\begin{proof}
The correspondence $\Phi:\mathcal{X}\to \mathcal{A}$ takes values in $\mathcal{A}$ by Proposition \ref{l inf is cartan}.  The correspondence $\Psi:\mathcal{A}\to \mathcal{X}$ takes values in $\mathcal{X}$ by Lemma \ref{cart coarse lem}.  

The statement about the composition $\Phi\circ \Psi$ follows immediately from Proposition \ref{equals roe}.

To see the stated property for $\Psi\circ \Phi$, we first show that $\mathcal{E}_A\subseteq \mathcal{E}$.  Given  $a\in \mathcal{N}_{C^*_u(X;\mathcal{E})}(\ell^\infty(X))$ and $\epsilon>0$, find $b\in \C_u[X;\mathcal{E}]$ such that $\|a-b\|<\epsilon$.  Therefore
\begin{equation}
\{(x,y)\mid |a_{xy}|\geq\epsilon\}\subseteq \{(x,y)\mid b_{xy}\neq 0\}\in\mathcal E,
\end{equation}
and hence $E_{a,\epsilon}\in\mathcal E$. Therefore $\mathcal{E}_A\subseteq \mathcal{E}$.

For the reverse inclusion, let $E$ be an arbitrary element of $\mathcal{E}$.  Lemma \ref{greedy} gives us a decomposition
\begin{equation}
E=\bigsqcup_{n=1}^N E_n
\end{equation}
of $E$ into sets $E_n$ whose intersection with each row and column contains at most one element.  Then with the notation of Definition \ref{p isoms}, $v^{E_n}$ is a well-defined partial isometry in $\mathcal N_{C^*_u(X;\mathcal E)}(\ell^\infty(X))$ for each $n$.  With the notation of Definition \ref{cart coarse def}, we have that $E_{v^{E_n},1/2}=E_n$, and thus $E_n$ is contained in $\mathcal{E}_A$.  As this is true for each $n$, $E$ is contained in $\mathcal{E}_A$, and we are done.
\end{proof} 

Finally we characterise when the coarse structure $(X,\mathcal E_A)$ is metrisable in terms of the Cartan pair $B\subseteq A$.  First, a general lemma.

\begin{lemma}\label{new cor}
With notation as in Lemma \ref{contained in roe}, assume moreover that $B$ is abstractly isomorphic to $\ell^\infty(I)$ for some set $I$. Then any normaliser $a\in \mathcal N_A(B)$ can be approximated arbitrarily well in norm by products $fv$, where $f\in B$ and $v$ is a partial isometry in $A$ normalising $B$.
\end{lemma}
\begin{proof}
Given any normaliser $c\in \mathcal N_{C^*_u(X,\mathcal E_A)}(\ell^\infty(X))$, and $\epsilon>0$, define $E:=\{(x,y)\in X\times X:|c_{xy}|\geq\epsilon\}$ so that $v^E$ is a partial isometry in $C^*_u(X,\mathcal E_A)$ normalising $\ell^\infty(X)$. Define 
\begin{equation}
f(x):=\begin{cases}c_{xy}&\exists y\in X,\ c_{xy}\neq 0\\0,&\text{otherwise}\end{cases}
\end{equation}
(the $y$ above being unique if it exists by Lemma \ref{linf normal} (i)), so that $\|c-fv^E\|\leq\epsilon$.  The result then transfers from $\ell^\infty(X)\subseteq C^*_u(X,\mathcal E_A)$ to $B\subseteq A$ by Proposition \ref{equals roe}.
\end{proof}

\begin{lemma}\label{more props}
Let $(A,B,\mathcal{K})$ and $(X,\mathcal{E})$ correspond to each other under the constructions of Definition \ref{const} and Theorem \ref{cart coarse prop}.  Then $B$ is co-separable in $A$ if and only if $\mathcal{E}_A$ is countably generated.  In particular, $B$ is co-separable in $A$ if and only if $\mathcal{E}_A$ is metrisable.
\end{lemma}

\begin{proof}
Suppose first that $\mathcal{E}_A$ is countably generated, say by $E^1,E^2,...$.  Then Lemma \ref{greedy} allows us to decompose each $E^m$ into finitely many parts 
\begin{equation}
E^m=\bigsqcup_{n=1}^{N_m} E^m_n
\end{equation}
such that each $E^m_n$ only intersects each row and column at most once.  Lemma \ref{u roe struc} then gives us a countable set of operators $\{v^{E^m_n}\mid m\geq 1, 1\leq n \leq N_m\}$ that together with $B\cong \ell^\infty(X)$ generate $A\cong C^*_u(X,\mathcal{E})$.  Hence $B$ is co-separable in $A$.

Conversely, suppose $B$ is co-separable in $A$. Using Lemma \ref{new cor}, we can find a countable set $S$ of partial isometries in $A$ normalising $B$, such that $C^*(S,B)=A$.  Moreover we may assume that $S$ is closed under taking finite products. Then, for $b_1,b_2\in B$ and $s_1,s_2\in S$, we have 
\begin{equation}
b_1s_1b_2s_2=b_1s_1s_1^*s_1b_2s_2=b_1(s_1b_2s_1^*)s_1s_2,
\end{equation}
which is of the form $bs$ for $b=b_1s_1b_2s_1^*\in B$ and $s=s_1s_2\in S$. As such, the collection of finite linear combinations $\{\sum_{i=1}^n b_is_i\mid b_i\in B,\ s_i\in S\}$ has dense linear span in $A$.  

Let $\mathcal D$ be the coarse structure generated by the countable family of sets $E_{s,1}:=\{(x,y)\mid |s_{xy}|\geq1\}=\{(x,y)\mid s_{xy}\neq 0\}$ indexed by $s\in S$. Each $E_{s,1}$ is in $\mathcal E_A$, so $\mathcal D\subseteq \mathcal E_A$. For the reverse inclusion, given a normaliser $a\in \mathcal N_A(B)$ and $\epsilon>0$, find a finite linear combination $\sum_{i=1}^nb_is_i$ with $b_i\in B$ and $s_i\in S$ such that $\|a-\sum_{i=1}^nb_is_i\|<\epsilon/2$.  Then
\begin{align}
\nonumber\{(x,y)\mid |a_{xy}|\geq\epsilon\}\subseteq&\bigcup_{i=1}^n\{(x,y)\mid |(b_is_i)_{xy}|\geq\epsilon/2n\}\\
\subseteq&\bigcup_{i=1}^n E_{s_i}\in \mathcal D.
\end{align}
Therefore $\mathcal E_A=\mathcal D$, and hence $\mathcal E_A$ is countably generated.

The remaining comment about metrisability is immediate from Theorem \ref{met the}.
\end{proof}

In the light of the previous results, it makes sense to encapsulate the key features of a Cartan pair which enable us to obtain a bounded geometry metric space in the following definition.

\begin{definition}\label{roe cartan def}
An inclusion  $B\subseteq A$ of $\mathrm{C}^*$-algebras is a \emph{Roe Cartan pair} if:
\begin{enumerate}[(i)]
\itemsep-.1cm 
\item $A$ is unital;
\item $A$ contains the $C^*$-algebra of compact operators on a separable infinite dimensional Hilbert space as an essential ideal\footnote{and therefore as its unique minimal ideal.};
\item $B$ is a co-separable Cartan subalgebra of $A$ abstractly isomorphic to $\ell^\infty(\mathbb N)$.
\end{enumerate}
A subalgebra $B$ of a uniform Roe algebra $C^*_u(X)$ is a \emph{Roe Cartan} subalgebra, if $B\subseteq C^*_u(X)$ is a Roe Cartan pair.
\end{definition}

With this definition, Theorem \ref{cart abs} is an immediate consequence of Theorem \ref{cart coarse prop} and Lemma \ref{more props}.

We end this section with a proof of Corollary \ref{IntroC}.  

\begin{proof}[Proof of Corollary \ref{IntroC}]
Assume that $X$ coarsely embeds into Hilbert space, and let $B\subseteq C^*_u(X)$ be a Roe Cartan.  Let $Y$ be the metric space associated to $B$ by Theorem \ref{cart abs}, so in particular $C^*_u(X)$ is isomorphic to $C^*_u(Y)$.  We now complete the proof by appealing to \cite[Corollary 1.5]{Braga:2019wv}, which states that if $X$ and $Y$ are bounded geometry metric spaces such that $X$ is coarsely embeddable in Hilbert space and $C^*_u(X)$ is isomorphic to $C^*_u(Y)$, then $X$ is coarsely equivalent to $Y$.
\end{proof}

\begin{remark}\label{bf ce rem}
Let $\mathcal{P}$ be a property of bounded geometry metric spaces that is invariant under coarse equivalences.  For the purposes of this remark, let us say that $\mathcal{P}$ implies \emph{rigidity} if whenever $X$ and $Y$ both have $\mathcal{P}$ and $C^*_u(X)$ is isomorphic to $C^*_u(Y)$, then $X$ is coarsely equivalent to $Y$.  On the other hand, let us say that $\mathcal{P}$ implies \emph{superrigidity} if whenever $X$ has $\mathcal{P}$ and $C^*_u(X)$ is isomorphic to $C^*_u(Y)$, then $X$ is coarsely equivalent to $Y$.  As is clear from the proof, any property $\mathcal{P}$ that implies superrigidity could be used as a hypothesis for Corollary \ref{IntroC} in place of coarse embeddability into Hilbert space.

Now, an earlier version of this paper proved Corollary \ref{IntroC} under the stronger assumption that $X$ has property A.  Indeed, property A implies rigidity by \cite[Theorem 1.4]{Spakula:2011bs}.  Moreover, property A for $X$ is equivalent to nuclearity of $C^*_u(X)$ by \cite[Theorem 5.3]{Skandalis:2002ng} or \cite[Theorem 5.5.7]{Brown:2008qy}, and it is clear that if $\mathcal{P}$ implies rigidity and is such that $\mathcal{P}$ for $X$ that can be characterized by an isomorphism invariant of $C^*_u(X)$, then in fact $\mathcal{P}$ implies superrigidity.

Earlier work of Braga and Farah \cite[Corollary 1.2]{Braga:2018dz} showed that coarse embeddability into Hilbert space implies rigidity.   This would not be enough for our proof of Corollary \ref{IntroC}, as it was not known if coarse embeddability into Hilbert space of $X$ can be characterised by an isomorphism-invariant of $C^*_u(X)$.  Given this, it is quite striking that Braga, Farah, and Vignati were able to prove that coarse embeddability into Hilbert space implies superrigidity in their very recent work \cite{Braga:2019wv}, as was used in the proof above.

Similar remarks to these also apply to Corollary \ref{IntroD}, which we discuss in the next section.
\end{remark}

\section{Uniqueness of Cartan subalgebras up to automorphism}\label{inn sec} 

In this short section we prove Corollary \ref{IntroD}.  This is a reasonably straightforward consequence of the results of the previous section combined with the main results of \cite{Braga:2018dz} and a theorem of Whyte \cite[Theorem 4.1]{Whyte:1999uq}. 

First, we give a slight variation of \cite[Theorem 4.1]{Whyte:1999uq}; this is probably well-known to experts.  Unexplained terminology in the proof can be found in the cited papers of Block and Weinberger, and of Whyte.

\begin{theorem}\label{whyte the}
Let $X$ and $Y$ be bounded geometry metric spaces, at least one of which is non-amenable.  Let $f:X\to Y$ be a coarse equivalence.  Then there is a bijective coarse equivalence from $X$ to $Y$ that is close to $f$.
\end{theorem}

\begin{proof}
For a bounded geometry metric space $Z$, let $H^{uf}_*(Z)$ denote the uniformly finite homology of $Z$ (with integer coefficients) in the sense of Block and Weinberger \cite[Section 2]{Block:1992qp}, and let $[Z]\in H^{uf}_0(Z)$ be the fundamental class of $Z$, i.e.\ the $0$-cycle defined by the constant function on $X$ with value one everywhere.  From the discussion around \cite[Proposition 2.1]{Block:1992qp}, if $f:X\to Y$ is a coarse embedding\footnote{``Coarse embedding'' is the current terminology for what Block and Weinberger call an \emph{effectively proper Lipschitz} map.}, then $f$ induces a map $f_*:H_*^{uf}(X)\to H^{uf}_*(Y)$.  Whyte proves in \cite[Theorem 4.1]{Whyte:1999uq} that if $f:X\to Y$ is a quasi-isometry between uniformly discrete\footnote{A metric space $X$ is \emph{uniformly discrete} if $\inf_{x,y\in X,x\neq y}d(x,y)>0$.}, bounded geometry metric spaces with $f_*[X]=[Y]$, then there is a bi-Lipschitz map $X\to Y$ that is close to $f$.  Let us sketch why Whyte's arguments also imply the result in the statement.

Now, with no real changes, Whyte's proof of \cite[Theorem 4.1]{Whyte:1999uq} as stated above shows that if $f:X\to Y$ is a map between bounded geometry (not necessarily uniformly discrete) metric spaces such that
\begin{enumerate}[(i)]
\itemsep-.1cm 
\item $f$ is a coarse embedding (so induces maps on $H^{uf}_*$),
\item $f_*[X]=[Y]$,
\item $f$ has coarsely dense image (meaning that $\sup_{x\in X}d(x,f(Y))<\infty$), and
\item there is a map $g:Y\to X$ such that $g\circ f$ and $f\circ g$ are close to the identities and $g$ has the properties (i), (ii), and (iii) above,
\end{enumerate}
then there is a bijection close to $f$.  Note, however, that if $f:X\to Y$ is a coarse equivalence, then it will have have properties (i), (iii) and (iv) above, and that any map close to a coarse equivalence is a coarse equivalence, so we get the following statement: if $f:X\to Y$ is a coarse equivalence between bounded geometry metric spaces with $f_*[X]=[Y]$, then there is a bijective coarse equivalence $X\to Y$ that is close to $f$. 

To complete the argument we must show that (ii) above is always satisfied under our hypotheses.  Indeed, note that amenability is invariant under coarse equivalence of bounded geometry metric spaces, as follows for example from \cite[Proposition 2.1 and Theorem 3.1]{Block:1992qp}.  Hence if one of $X$ or $Y$ as in our set up is non-amenable, then the other is.  Moreover, Block and Weinberger show in \cite[Theorem 3.1]{Block:1992qp} that $X$ is non-amenable if and only if $H^{uf}_0(X)=0$, and thus condition (ii) from Whyte's theorem is vacuous in our set-up.  This completes the proof. 
\end{proof}

\begin{proof}[Proof of Corollary \ref{IntroD}]
Let $Y$ be as in Theorem \ref{cart abs}, and identify $B$ with $\ell^\infty(Y)$, and $A=C^*_u(X)$ with $C^*_u(Y)$.  Choose an orthonormal basis $\{\xi_y\}_{y\in Y}$ for $H=\ell^2(X)$ that is compatible with the identification $B\cong \ell^\infty(Y)$. Precisely, if the minimal projection in $B$ corresponding to the characteristic function of $\{y\}$ is $p_y$, then choose $\xi_y$ to be a unit vector in the image of $p_y$.   

As $X$ coarsely embeds into Hilbert space, $X$ and $Y$ are coarsely equivalent by \cite[Corollary 1.5]{Braga:2019wv}, just as in the proof of Corollary \ref{IntroC}.  Hence Theorem \ref{whyte the} gives us a bijective coarse equivalence $f:X\to Y$.  

Now define a map $u:\ell^2(X)\to \ell^2(X)$ by $u\delta_x=\xi_{f(x)}$. This is a unitary isomorphism, as $f$ is a bijection.  Using that $f$ is a coarse equivalence, it follows that conjugation by $u$ takes $A$ to $C^*_u(Y)$, or in other words, $u$ conjugates $A$ to itself.  Define $\alpha:A\to A$ by $\alpha(a)=uau^*$; we then have $\alpha(\ell^\infty(X))=u\ell^\infty(X)u^*=B$ as required. 
\end{proof}

\begin{remark}\label{inner rem}
It does not seem to be clear if the unitary $u$ produced by the above proof is actually in $A$ (or can be chosen to be in $A$), so we cannot conclude that the automorphism $\alpha$ in the statement of Corollary \ref{IntroD} is inner.  Note that while any automorphism of a uniform Roe algebra $C^*_u(X)$ is induced by a unitary $u:\ell^2(X)\to \ell^2(X)$ (see \cite[Lemma 3.1]{Spakula:2011bs}) there are often many non-inner automorphisms of $C^*_u(X)$.  For an illustrative example, take $X=\Z$ and the automorphism given by conjugation by the unitary
$$
u:\ell^2(\Z)\to \ell^2(\Z), \quad \delta_n\mapsto \delta_{-n}.
$$
This is not inner, or even approximately inner.  One can see this, for example, as it is non-trivial on $K$-theory.  Indeed, the Pimsner-Voiculescu sequence implies that $K_1(C^*_u(X))$ is isomorphic to $\Z$, and generated by the class $[v]$ of the bilateral shift on $\ell^2(\Z)$.  We have $uvu^*=v^*$, so conjugation by $u$ takes $[v]$ to $-[v]$; on the other hand, if the corresponding automorphism were (approximately) inner, then it would act trivially on $K$-theory.
\end{remark}

\section{Uniqueness of Cartan subalgebras up to inner automorphism}\label{c*-alg sec}

In this section, we prove our main result, Theorem \ref{c*-alg the}.

The following notation will be in force for the rest of the section.  Let $B\subseteq C^*_u(X)$ satisfy the assumptions of Theorem \ref{c*-alg the}.  Theorem \ref{cart abs} (with $H=\ell^2(X)$) implies that there is a bounded geometry metric space $Y$, and a unitary isomorphism $v:\ell^2(Y)\to \ell^2(X)$ such that 
\begin{equation}
v\ell^\infty(Y)v^*=B \quad \text{and}\quad vC^*_u(Y)v^*=C^*_u(X).
\end{equation}

\begin{lemma}\label{y fdc}
With notation as above, the space $Y$ has property A.
\end{lemma}

\begin{proof}
This follows as a bounded geometry metric space has property A if and only if its uniform Roe algebra is nuclear by \cite[Theorem 5.3]{Skandalis:2002ng} or \cite[Theorem 5.5.7]{Brown:2008qy}.
\end{proof}

At this point, we have two spaces $X$ and $Y$ with property A, and a unitary isomorphism $v:\ell^2(Y)\to \ell^2(X)$ that conjugates $C^*_u(Y)$ to $C^*_u(X)$.  Our task is to show that there is some unitary $u\in C^*_u(X)$ that conjugates $v\ell^\infty(Y)v^*$ to $\ell^\infty(X)$.

We will need some more notation that will be used throughout the rest of this section.  For each $y\in Y$, let $q_y\in \mathcal{B}(\ell^2(Y))$ denote the orthogonal projection onto the span of $\delta_y$.  Similarly, for each $x\in X$, let $p_x\in \mathcal{B}(\ell^2(X))$ be the orthogonal projection onto the span of $\delta_x$. For a subset $C$ of $X$ (respectively, of $Y$) define 
\begin{equation}
p_C:=\sum_{x\in C} p_x \quad \Big(\text{respectively, } q_C:=\sum_{y\in C} q_y\Big)
\end{equation}
for the corresponding multiplication operator on $\ell^2(X)$ (respectively, on $\ell^2(Y)$).

The proof splits fairly cleanly into three main steps. 

\begin{enumerate}
\item Uniform approximability.  For each subset $C$ of $Y$, we know that $vq_Cv^*$ is in the uniform Roe algebra of $X$, whence the following holds: ``$\forall \epsilon>0$, $\forall C\subseteq Y$, $\exists s>0$ such that $v\chi_Cv^*$ can be approximated within $\epsilon$ by an operator in $C^*_u(X)$ with propagation at most $s$''. A priori $s$ depends on $\epsilon$ and $C$.  Our first aim is to improve this statement, so that $s$ only depends on $\epsilon$. Using a result of Braga and Farah (\cite[Lemma 4.9]{Braga:2018dz}), this can be achieved with no assumptions on $X$ and $Y$ beyond that they are bounded geometry metric spaces.
\item The operator norm localisation property.  The operator norm localisation property was introduced by Chen, Tessera, Wang, and Yu \cite{Chen:2008so}; it was shown to be equivalent to property A by Sako \cite{Sako:2012fk}.  The key application here is roughly the following statement: ``$\forall\epsilon>0$, $\exists r>0$  such that  $\forall y\in Y$,  $\exists X_y\subseteq X$ of diameter at most $r$ such that $\|vq_yv^*p_{X_y}\|\geq 1-\epsilon$''.  This says roughly that we can match points in $Y$ to uniformly bounded subsets of $X$.  We need a stronger, somewhat more quantitative version of this that also works for subsets of $X$ other than singletons $\{x\}$; see Lemma \ref{soup up}.
\item Completion of the proof.  To finish the proof, the above step can be combined with Hall's marriage theorem to get an injection $f:Y\to X$ with $f(y)\in X_y$ for all $x$.  As the situation is symmetric, we get a similar injection $g:X\to Y$, and so a bijection $h:X\to Y$ from K\"{o}nig's proof of the Cantor-Schr\"{o}der Bernstein theorem.  This $h$ defines a unitary $w:\ell^2(X)\to \ell^2(Y)$ by $w\delta_x=\delta_{h(x)}$, which conjugates $\ell^\infty(X)$ to $\ell^\infty(Y)$.  To complete the proof, it suffices to show that $u:=vw$ is contained in $C^*_u(X)$: this is achieved by using the quantitative results from the previous step to get a weak form of finite propagation for $u$, and then appealing to an approximation result due to \v{S}pakula and Zhang \cite{Spakula:2018aa} (which builds on work of \v{S}pakula and Tikuisis \cite[Theorem 2.8]{Spakula:2017aa}) to show that this weak property is enough.  
\end{enumerate}

\subsubsection*{Step one: uniform approximability}

Here is the result of Braga and Farah that we will use; it is a special case of \cite[Lemma 4.9]{Braga:2018dz}.  

\begin{lemma}\label{bf the}
Let $Z$ be a bounded geometry metric space.  Suppose $(a_n)_{n=1}^\infty$ is a sequence of finite rank operators on $\ell^2(Z)$ such that for every bounded sequence $(\lambda_n)_{n=1}^\infty$ of complex numbers, the series $\sum_{n=1}^\infty \lambda_n a_n$ converges strongly to an operator in $C^*_u(Z)$.  Then for every $\epsilon>0$ there exists $s>0$ such that for every bounded sequence $(\lambda_n)_{n=1}^\infty$ there is $a\in C^*_u(X)$ of propagation at most $s$ and $\|\sum_{n=1}^\infty \lambda_n a_n-a\|<\epsilon$. \qed
\end{lemma}

The content of the lemma is in the order of quantifiers. The corresponding statement with $s$ depending also on the bounded sequence $(\lambda_n)_{n=1}^\infty$ is immediate.  Here is the consequence we need.
\begin{corollary}\label{spti}
For any $\epsilon>0$ there exists $s>0$ such that for any $D\subseteq Y$ there is $a\in C^*_u(X)$ with propagation at most $s$, and $\|vq_Dv^*-a\|\leq \epsilon$.
\end{corollary}

\begin{proof}
The family $\{vq_yv^*\}_{y\in Y}$ has the property that for any bounded sequence $(\lambda_y)_{y\in Y}$ of complex numbers, $\sum_{y\in Y}\lambda_yvq_yv^*$ converges strongly to $v^*(\sum_{y\in Y}\lambda_yq_y)v \in C^*_u(X)$.  The corollary is then immediate from Lemma \ref{bf the}.
\end{proof}

\subsubsection*{Step two: the operator norm localisation property}

We now recall the definition of the operator norm localisation property (ONL) from \cite[Definition 2.2]{Chen:2008so}.  The version we give below is equivalent to the usual one by \cite[Proposition 2.4]{Chen:2008so}. Using the main result of  \cite{Sako:2012fk}, property A is equivalent to ONL, so both our spaces $X$ and $Y$ have ONL by Lemma \ref{y fdc}.

\begin{definition}\label{onl}
A bounded geometry metric space $Z$ has the \emph{operator norm localisation property} (ONL) if for any $\epsilon\in (0,1)$ and any $s>0$ there is $r>0$ such that for any operator $a\in C^*_u(Z)$ with propagation at most $s$ there exists a unit vector $\xi\in \ell^2(Z)$ with 
\begin{equation}
\|a\xi\|\geq (1-\epsilon)\|a\|
\end{equation}
and with $\xi$ supported in a set of diameter at most $r$. 
\end{definition}

Again, the point is order of quantifiers; with $r$ also depending on $a$ the analogous statement is automatic.   \

\begin{lemma}\label{msp cor}
\begin{enumerate}[(i)]
\itemsep-.1cm 
\item \label{msp cor spec} For any $\epsilon\in (0,1)$ there exists $r>0$ such that for any non-empty $D\subseteq Y$ there is $E\subseteq X$ with $\text{diam}(E)\leq r$ and 
\begin{equation}
\|vq_Dv^*p_{E}\|\geq  (1-\epsilon).
\end{equation}
\item \label{msp cor gen} For any $\epsilon\in (0,1)$ there exists $r>0$ such that for any $C\subseteq X$ and $D\subseteq Y$ there is $E\subseteq X$ with $\text{diam}(E)\leq r$ and 
\begin{equation}
\|vq_Dv^*p_{C\cap E}\|\geq  (1-\epsilon)\|vq_Dv^*p_C\|-\epsilon.
\end{equation}
\end{enumerate}
\end{lemma}

\begin{proof}
We look at part \eqref{msp cor gen} first. Fix $\epsilon>0$.
Using Corollary \ref{spti} there exists $s>0$ (depending only on $\epsilon$) such that for any $D\subseteq Y$, there is $a_0\in C^*_u(X)$ with propagation at most $s$ such that $\|vq_Dv^*-a_0\|<\epsilon/2$.  As $p_C$ has propagation zero, it follows that if $a:=a_0p_C$ then $a$ still has propagation at most $s$, and as $\|p_C\|\leq 1$ we have that
\begin{equation}\label{from spti}
\|vq_Dv^*p_C-a\|<\epsilon/2.
\end{equation}
Using the operator norm localisation property, there exists $r>0$ (depending only on $s$ and $\epsilon$) such that there is a unit vector $\xi\in \ell^2(X)$ with support in a set $E\subseteq X$ of diameter at most $r$ such that $\|a\xi\|\geq (1-\epsilon)\|a\|$.  Hence in particular we get
\begin{equation}\label{from onl}
\|ap_E\|\geq (1-\epsilon)\|a\|.
\end{equation}
Now, from line \eqref{from spti} we have
\begin{equation}
\|vq_Dv^*p_{C}p_{E}-ap_E\|<\epsilon/2.
\end{equation}
As $p_Cp_E=p_{C\cap E}$, this implies that
\begin{equation}
\|vq_Dv^*p_{C \cap E}\|> \|ap_E\|-\epsilon/2.
\end{equation}
Combining this with line \eqref{from onl} gives
\begin{equation}
\|vq_Dv^*p_{C \cap E}\|> (1-\epsilon)\|a\|-\epsilon/2,
\end{equation}
and applying line \eqref{from spti} again gives
\begin{align}
\|vq_Dv^*q_{C \cap E}\|&> (1-\epsilon)\|vp_Dv^*q_{C}\|-(1-\epsilon)\epsilon/2-\epsilon/2\\
&>(1-\epsilon)\|vp_Dv^*q_{C}\|-\epsilon,
\end{align}
proving \eqref{msp cor gen}.

Part \eqref{msp cor spec} follows immediately from part \eqref{msp cor gen}.  Indeed, let $r>0$ be as in the statement of part \eqref{msp cor gen} for the `error parameter' $\epsilon/2$, and take $C=X$. 
\end{proof}

We can interchange the roles of $X$ and $Y$ in the previous argument, leading to the following lemma.

\begin{lemma}\label{msp cor 2}
For any $\epsilon\in (0,1)$ there exists $r>0$ such that for any $C\subseteq X$ and any $D\subseteq Y$, there is $F\subseteq Y$ with 
\begin{equation}
\|vq_{D\cap F}v^*p_C\|\geq (1-\epsilon)\|vq_Dv^*p_C\|-\epsilon. 
\end{equation}
and $\text{diam}(F)\leq r$. 
\end{lemma}
\begin{proof}
In the previous lemma we regarded $v^*\ell^\infty(Y)v$ as an `exotic Cartan' in $C^*_u(X)$, but we could equally well regard $v\ell^\infty(X)v^*$ as an exotic Cartan in $C^*_u(Y)$. As $Y$ also has ONL, we obtain that for every $\epsilon\in(0,1)$, there exists $r>0$ such that for any $C\subseteq X$ and $D\subseteq Y$, there exists $F\subseteq Y$ of diameter at most $r$ such that
\begin{equation}
\|vp_Cv^*q_{D\cap F}\|\geq (1-\epsilon)\|vp_Cv^*q_D\|.
\end{equation}
The result follows as $\|vp_Cv^*q_{D\cap F}\|=\|q_{D\cap F}vp_Cv^*\|=\|vq_{D\cap F}v^*p_C\|$, and likewise for the right hand side.
\end{proof}

We need some more notation.  For each $y\in Y$ and $\delta>0$, define 
\begin{equation}\label{xydelta def}
X_{y,\delta}:=\{x\in X\mid \|vq_yv^*p_x\|^2\geq \delta\}.
\end{equation}
Analogously, define 
\begin{equation}\label{yxdelta def}
Y_{x,\delta}:=\{y\in Y\mid \|v^*p_xvq_y\|^2\geq \delta\}.
\end{equation}
One should think of $X_{y,\delta}$ as being the part of $X$ that is `$\delta$-close' to $y$ in some sense, and similarly for $Y_{x,\delta}$.  We extend these notions to sets by defining $X_{D,\delta}:=\bigcup_{y\in D}X_{y,\delta}$ for $D\subseteq X$, and $Y_{C,\delta}:=\bigcup_{x\in C}X_{c,\delta}$ for $C\subseteq X$.

\begin{lemma}\label{delt cor}
With notation as in lines \eqref{xydelta def} and \eqref{yxdelta def} as above:
\begin{enumerate}[(i)]
\itemsep-.1cm 
\item \label{near xydelta} for each $\epsilon>0$, there is $\delta>0$ such that for all $y\in Y$, $\|vq_yv^*p_{X_{y,\delta}}\|^2\geq 1-\epsilon$;
\item \label{xydelta bound} for each $\delta>0$ there exists $r>0$ such that for all $y\in Y$, the diameter of $X_{y,\delta}$ is at most $r$.
\end{enumerate}
\end{lemma}

\begin{proof}
Applying part \eqref{msp cor spec} of Lemma \ref{msp cor} with $D=\{y\}$, there is $r>0$ depending only on $\epsilon$ such that for each $y\in Y$ there is $E\subseteq X$ with $\text{diam}(E)\leq r$ and 
\begin{equation}\label{easy est}
\|vq_yv^*p_E\|^2\geq  1-\epsilon/2.
\end{equation}
Let $\xi_y$ be any unit vector in the range of the rank one projection $vq_yv^*$, and note that 
\begin{equation}
\|vq_yv^*p_E\|^2=\|p_Evq_yv^*\|^2=\|p_E\xi_y\|^2,
\end{equation}
so line \eqref{easy est} above says that 
\begin{equation}\label{sum of xiy}
\sum_{x\in E}|\xi_y(x)|^2\geq 1-\epsilon/2.
\end{equation}
Notice that this implies that $X_{y,\epsilon/2}\subseteq E$, proving (\ref{xydelta bound}) after relabelling $\epsilon$ as $2\delta$, as otherwise the sum above differs from $1=\sum_{x\in X}|\xi_y(x)|^2$ by a term of size at least $\epsilon/2$, a contradiction.

Let $N\in \N$ be an absolute bound on the cardinalities of all balls of radius $r$ in $X$, and let $\delta<\frac{\epsilon}{2N}$ (which only depends on $r$ and $\epsilon$, so only on $\epsilon$).  Then 
\begin{align}
\|vq_yv^*p_{X_{y,\delta}}\|^2 & =\sum_{x\in X_{y,\delta}}|\xi_y(x)|^2 \nonumber \\ & \geq \sum_{x\in E\cap X_{y,\delta}}|\xi_y(x)|^2 \nonumber \\
& =\sum_{x\in E}|\xi_y(x)|^2-\sum_{x\in E\setminus X_{y,\delta}}|\xi_y(x)|^2.
\end{align}
Now, on the one hand line \eqref{sum of xiy} gives $\sum_{x\in E}|\xi_y(x)|^2\geq 1-\epsilon/2$, and on the other hand $|\xi_y(x)|^2=\|p_xvq_yv^*\|^2<\delta$ for all $x\not\in X_{y,\delta}$.  Moreover, $|E|\leq N$ whence $\sum_{E\setminus X_{y,\delta}}|\xi_y(x)|^2<N\delta$.  The previous displayed inequality thus implies
\begin{equation}
\|vq_yv^*\chi_{X_{y,\delta}}\|^2\geq 1-\frac{\epsilon}{2}-N\delta,
\end{equation}
and the right hand side is at least $1-\epsilon$ by choice of $\delta$, proving (\ref{near xydelta}).
\end{proof}

We now bootstrap Lemma \ref{delt cor} \eqref{near xydelta} to subsets.

\begin{lemma}\label{soup up}
For any $\epsilon>0$, there is $\delta>0$ such that for any subset $D$ of $Y$, 
\begin{equation}
\|vq_Dv^*(1-p_{X_{D,\delta}})\|<\epsilon,
\end{equation}
and for any subset $C$ of $X$
\begin{equation}
\|v^*p_Cv(1-q_{Y_{C,\delta}})\|<\epsilon.
\end{equation}
\end{lemma}

\begin{proof}
Fix $\gamma>0$, to be chosen later in a way depending only on $\epsilon$.  Using Lemma \ref{msp cor 2} there is $r>0$ such that for any $C\subseteq X$ and $D\subseteq Y$ we have $F\subseteq Y$ with $\text{diam}(F)\leq r$ such that 
\begin{equation}
\|vq_{D\cap F}v^*p_C\Big\|\geq (1-\gamma)\|vq_Dv^*p_C\|-\gamma.
\end{equation}
Hence 
\begin{equation}\label{start}
\|vq_Dv^*p_C\|\leq \frac{\|vq_{D\cap F}v^*p_C\|+\gamma}{1-\gamma}.
\end{equation}
Let $M$ be some large positive number, to be chosen later (in a way that depends only on $r$ and $\gamma$, so only on $\gamma$, so only on $\epsilon$).  Applying Lemma \ref{delt cor} \eqref{near xydelta} gives $\delta_Y>0$ such that $\|vq_yv^*p_{X_{y,\delta_Y}}\|^2\geq 1-\frac{1}{M}$ for all $y\in Y$.  As before, let $\xi_y$ be any unit vector in the image of $vq_yv^*$, and note that 
\begin{equation}
\|vq_yv^*p_{X_{y,\delta_Y}}\|^2+\|vq_yv^*(1-p_{X_{y,\delta_Y}})\|^2=\sum_{x\in X_{y,\delta_Y}}|\xi_y(x)|^2 +\sum_{x\not\in X_{y,\delta_Y}}|\xi_y(x)|^2,
\end{equation}
which equals one, and therefore also that 
\begin{equation}
\|vq_yv^*(1-p_{X_{y,\delta_Y}})\|\leq 1/\sqrt{M}
\end{equation}
for all $y\in Y$.  Hence, for any $y\in D$
\begin{equation}\label{partway}
\|vq_yv^*(1-p_{X_{D,\delta_Y}})\|\leq \|vq_yv^*(1-p_{X_{y,\delta_Y}})\|\leq1/\sqrt{M}.
\end{equation}  
Now, take $C:=X\setminus X_{D,\delta_Y}$, and find $F\subseteq Y$ with diameter at most $r$, so that line (\ref{start}) holds.  Using this and line \eqref{partway}, we get
\begin{align}
\|vq_Dv^*(1-p_{X_{D,\delta_Y}})\| & \leq \frac{|F|\sup_{y\in D}\|vq_yv^*(1-p_{X_{D,\delta_Y}})\|+\gamma}{1-\gamma} \nonumber \\ & \leq \frac{|F|\frac{1}{\sqrt{M}}+\gamma}{1-\gamma}. 
\end{align}
Let $N$ be a bound on the cardinalities of all $r$-balls in $Y$, and set $M:=(N/\gamma)^2$. Then the above says that 
\begin{equation}
\|vq_Dv^*(1-p_{X_{D,\delta_Y}})\|\leq \frac{2\gamma}{1-\gamma}.
\end{equation}
Choosing $\gamma<\frac{\epsilon}{2+\epsilon}$, this proves the first claim of the lemma.

Interchanging the roles of $X$ and $Y$, exactly as in Lemma \ref{msp cor 2}, we can run the proof of Lemma \ref{delt cor} and the proof above, to obtain $\delta_X>0$ such that $\|v^*p_Cv(1-q_{Y_{C,\delta_X}})\|<\epsilon$ for all $C\subseteq X$.  Then, we take $\delta=\min(\delta_X,\delta_Y)$. 
\end{proof}

\subsubsection*{Step 3: completion of the proof}

To complete the proof, we first give an application of Hall's marriage theorem to construct appropriate maps.

\begin{lemma}\label{hall lem}
There exist $\delta>0$ and injections $f:Y\to X$ and $g:X\to Y$ such that $f(y)\in X_{y,\delta}$ for all $y\in Y$ and $g(x)\in Y_{x,\delta}$ for all $x\in X$.
\end{lemma}

\begin{proof}
Fix $\epsilon=1/2$ (any $\epsilon<1$ would work), and let $\delta$ satisfy the condition in Lemma \ref{soup up} for this $\epsilon$.  We first claim that for any finite $D\subseteq Y$, the cardinality of $X_{D,\delta}$ is at least as large as that of $D$, or in other words that that the rank of $p_{X_{D,\delta}}$ is at least as big as that of $vq_Dv^*$.  If not, then the rank of $vq_Dv^*$ is strictly larger than that of $p_{X_{D,\delta}}$; this forces the images of $vq_Dv^*$ and $1-p_{X_{D,\delta}}$ to have non-trivial intersection and thus $\|vq_Dv^*(1-p_{X_{D,\delta}})\|\geq 1$, contradicting the inequality in the first statement of Lemma \ref{soup up}.

Consider now the function $\phi:Y\to \mathcal{P}(X)$ defined by $\phi(y):=X_{y,\delta}$. Then for any finite subset $D\subseteq Y$, \begin{equation}
\Big|\bigcup_{y\in D}\phi(y)\Big|=|X_{D,\delta}|\geq |D|.
\end{equation}
The existence of $f$ follows from Hall's marriage theorem.

The existence of $g$ follows in exactly the same way, using the second statement in Lemma \ref{soup up}.
\end{proof}

\begin{corollary}\label{csb cor}
There exists $\delta>0$, $r>0$ and a bijection $h:X\to Y$ such that for any $x\in X$, $X_{h(x),\delta}$ is contained in the ball $B(x;r)$ around $x$ of radius $r$. 
\end{corollary}

\begin{proof}
Let $\delta>0$ and $f:Y\to X$ and $g:X\to Y$ be injections as in Lemma \ref{hall lem}.  K\"{o}nig's proof of the Cantor-Scr\"{o}der-Bernstein theorem gives us a bijection $h:X\to Y$ with the property that for each $x\in X$, either $h(x)=g(x)$, or $x$ is in the image of $f$ and $h(x)=f^{-1}(x)$.  

To complete the proof we must show that there exists $r>0$ such that for every $x\in X$, $X_{h(x),\delta}$ is contained in the ball $B(x;r)$ centered at $x$ with radius $r$.  Indeed, let $r$ equal the supremum of the diameters of the sets $X_{y,\delta}$ as $y$ ranges over $Y$; $r$ is finite by part (ii) of Lemma \ref{delt cor}.    Note first that if $x\in X$ is such that $h(x)=f^{-1}(x)$ for some $x\in X$, then $f(h(x))=x$ is an element of $X_{h(x),\delta}$ by the properties of $f$.  This implies that $X_{h(x),\delta}$ is contained in $B(x;r)$ by choice of $r$.  On the other hand, say $x\in X$ is such that $h(x)=g(x)$.  Then, by the defining property of $g$, $g(x)\in Y_{x,\delta}$, from which it follows that $\|v^*p_xvq_{g(x)}\|^2\geq \delta$.  Hence $\|p_xvq_{g(x)}v^*\|^2\geq \delta$, which says exactly that $x$ is in $X_{g(x),\delta}$.  The result follows by assumption on the diameter of all of the $X_{y,\delta}$.
\end{proof}

Now let $h:X\to Y$ be any bijection as in the conclusion of Corollary \ref{csb cor} for some appropriate $\delta>0$.  Let $w:\ell^2(X)\to \ell^2(Y)$ be the unitary defined by $w\delta_x=\delta_{h(x)}$.  Clearly then $w^*\ell^\infty(Y)w=\ell^\infty(X)$, and as also $v\ell^\infty(Y)v^*=B$ we thus get that $w^*v^*Bvw=\ell^\infty(X)$.  To complete the proof, it suffices to show that the unitary 
\begin{equation}\label{the one}
u:=w^*v^*
\end{equation}
is in $C^*_u(X)$.  To this end, we need a general criterion of \v{S}pakula and Zhang \cite{Spakula:2018aa} (itself building on an earlier result of \v{S}pakula and Tikuisis \cite[Theorem 2.8]{Spakula:2017aa}). 

\begin{theorem}[\v{S}pakula and Zhang]\label{st the}
Let $Z$ be a bounded geometry metric space with property A, and let $a\in \mathcal{B}(\ell^2(Z))$ be such that for any $\epsilon>0$ there exists $s>0$ such that if $C,D\subseteq Z$ satisfy $d(C,D)>s$, then $\|\chi_Ca\chi_D\|<\epsilon$.  Then $a$ is in $C^*_u(X)$.
\end{theorem}

We finally have all the ingredients in place to prove our main result.

\begin{proof}[Proof of Theorem \ref{c*-alg the}]
Let $h$ and $\delta$ be as in the conclusion of Corollary \ref{csb cor}, and let $u$ be the associated unitary as in line \eqref{the one}.  We claim first that for any $\epsilon>0$ there is $t>0$ such that for any subset $C\subseteq X$ if $N_t(C):=\{x\in X\mid d(x,C)\leq t\}$ then we have 
\begin{equation}
\|vq_{h(C)}v^*(1-p_{N_t(C)})\|<\epsilon.
\end{equation}
Indeed, applying Lemma \ref{soup up} with $D:=h(C)$ gives us $\gamma>0$ such that 
\begin{equation}
\|vq_{h(C)}v^*(1-p_{X_{h(C),\gamma}})\|<\epsilon.
\end{equation}
Now, we may assume that $\gamma\leq \delta$ and thus we have that $X_{h(x),\gamma}\supseteq X_{h(x),\delta}$ for all $x\in C$.  Let $r$ be such that $X_{h(x),\delta}$ is contained in $B(x;r)$ (such exists by Corollary \ref{csb cor}), and let $s$ be such that $X_{h(x),\gamma}$ has diameter at most $s$ for all $x\in C$ (such an $s$ exists by Lemma \ref{delt cor}, part \eqref{xydelta bound}).  Hence each $X_{h(x),\gamma}$ is contained in $B(x;s+r)$.  The claim follows with $t=s+r$.  

Now, from the claim we have that for any $\epsilon>0$ there is $t>0$ such that for any subset $C\subseteq X$ we have
\begin{equation}
\|vwp_{C}w^*v^*(1-p_{N_t(C)})\|<\epsilon.
\end{equation}
Hence for any $\epsilon>0$ there is $t>0$ such that for any subset $C\subseteq X$ we have
\begin{equation}
\|p_Cu(1-p_{N_t(C)})\|<\epsilon,
\end{equation}
and this in turn implies that for any $\epsilon>0$ there is $t>0$ such that for any subsets $C,D\subseteq X$ with $d(C,D)>t$ we have that 
\begin{equation}
\|\chi_Cu\chi_D\|<\epsilon.
\end{equation}
Hence, by Theorem \ref{st the}, $u$ is in $C^*_u(X)$.
\end{proof}

\begin{remark}\label{fdc vs a rem}
An earlier version of this paper used \emph{finite decomposition complexity} (FDC) in place of the hypothesis of property A  in Theorem \ref{c*-alg the}.  FDC was introduced by Guentner, Tessera, and Yu \cite{Guentner:2009tg} in their work on topological rigidity.  We initially used FDC as we originally appealed to an analogue of Theorem \ref{st the}, due to \v{S}pakula and Tikuisis \cite{Spakula:2017aa}, that has FDC as a hypothesis in place of property A.  Since the first version of our paper, \v{S}pakula and Zhang  built on the techniques of \cite{Spakula:2017aa} (and introduced some others) to prove Theorem \ref{st the} in the form stated above.  Property A is known to be implied by FDC \cite[Theorem 4.6]{Guentner:2013aa}, so the version of Theorem \ref{c*-alg the} with property A as a hypothesis is a priori stronger than the version with FDC.  Note though that it remains an open problem whether property A implies FDC.
\end{remark}

We end the paper with the following `rigidity' corollary of Theorem \ref{c*-alg the}.

\begin{corollary}\label{rig the}
Say $X$ and $Y$ are bounded geometry metric spaces.  Then the following are equivalent:
\begin{enumerate}[(i)]
\itemsep-.1cm 
\item there is a bijective coarse equivalence between $X$ and $Y$;
\item the coarse groupoids associated to $X$ and $Y$ (see \cite{Skandalis:2002ng} or \cite[Chapter 10]{Roe:2003rw}) are isomorphic;
\item there is a $^*$-isomorphism from $C^*_u(X)$ to $C^*_u(Y)$ that takes $\ell^\infty(X)$ to $\ell^\infty(Y)$.
\end{enumerate}
Moreover, if $X$ has property A, then these statements are equivalent to
\begin{enumerate}
\item[(iv)] there is a $^*$-isomorphism from $C^*_u(X)$ to $C^*_u(Y)$.
\end{enumerate}
\end{corollary}

The equivalence of (i), (ii), and (iii) in the above is fairly well-known: it seems to have been observed independently by several people.  We are not sure if it has explicitly appeared in the literature before: see \cite[Theorem 8.1]{Braga:2018dz} for a closely related, and overlapping, result.  The content of the corollary is the equivalence of these with (iv) when $X$ has property A.  

\begin{proof}[Proof of Corollary \ref{rig the}]
The fact that (i) implies (ii) implies (iii) is straightforward.  The implication (iii) implies (i) follows as such a $^*$-isomorphism induces a bijection between the minimal projections in $\ell^\infty(X)$ and those in $\ell^\infty(Y)$, and thus a bijection $f:X\to Y$.  We claim that $f$ is uniformly expansive in the sense of Definition \ref{bg def}.  Indeed, if not, then there is $r>0$ and a sequence $\big((x^{(n)}_1,x^{(n)}_2)\big)_{n=1}^\infty$ of pairs in $X\times X$ such that $d_X(x^{(n)}_1,x^{(n)}_2)\leq r$ for all $n$, but such that $d_Y(f(x^{(n)}_1),f(x^{(n)}_2))\to\infty$ as $n\to\infty$.  Passing to a subsequence and using bounded geometry, we may assume that no point of $X$ appears twice in the set $\{x^{(n)}_1,x^{(n)}_2\mid n\in \N\}$.  Now, consider the $X\times X$ matrix defined by the condition that $a_{x^{(n)}_1x^{(n)}_2}=1$ for all $n$, and all other matrix entries zero.  Our assumptions that no element appears twice in $\{x^{(n)}_1,x^{(n)}_2\mid n\in \N\}$ implies that this matrix is supported on a single diagonal, and thus defined a bounded operator $a$ on $\ell^2(X)$ by Lemma \ref{linf normal}.  Moreover, the fact that $d(x^{(n)}_1,x^{(n)}_2)\leq r$ for all $n$ implies that $a$ is in $\C_u[X]$.  On the other hand, our isomorphism takes $a$ to an operator in $C^*_u(Y)$ whose $(f(x^{(n)}_1),f(x^{(n)}_2))^\text{th}$ matrix entry is one for all $n$.  The assumption that $d_Y(f(x^{(n)}_1),f(x^{(n)}_2))\to\infty$ implies that this is impossible, however.  A precisely analogous argument now shows that $f^{-1}$ is also uniformly expansive, so $f$ is a coarse equivalence as required.

As (iii) implies (iv) is trivial, to complete the proof it suffices to prove (iv) implies (iii).  Assume that $C^*_u(X)$ and $C^*_u(Y)$ are $^*$-isomorphic.  As in \cite[Lemma 3.1]{Spakula:2011bs}, there is a unitary isomorphism $v:\ell^2(X)\to \ell^2(Y)$ such that $vC^*_u(X)v^*=C^*_u(Y)$.  Let $B=v\ell^\infty(X)v^*$, which satisfies the assumptions of Theorem \ref{c*-alg the}; this needs that $Y$ has property A, which follows from \cite[Theorem 1.4]{Spakula:2011bs} and the fact that property A is a coarse invariant.  Hence there is $u\in C^*_u(Y)$ with $uBu^*=\ell^\infty(Y)$.  Now, we have that $uv\ell^\infty(X)v^*u^*=\ell^\infty(Y)$.  Then $\mathrm{ad}(uv)$ is an isomorphism from $C^*_u(X)$ onto $C^*_u(Y)$ mapping $\ell^\infty(X)$ onto $\ell^\infty(Y)$. \end{proof}

\begin{remark}\label{a vs ce}
Using Corollary \ref{IntroD}, one could replace the assumption of property A in the above by the assumption that $X$ coarsely embeds into Hilbert space, and is non-amenable.  The proof is essentially the same.  Note that for groups, amenability implies property A (which implies coarse embeddability into Hilbert space).  Therefore if $X$ is restricted to the class of groups, we can replace property A in Theorem \ref{rig the} by coarse embeddability into Hilbert space with no (non-)amenability assumption.  However, for general bounded geometry metric spaces, amenability does not imply property A (or even coarse embeddability into Hilbert space), so we cannot get away with this in general.  
\end{remark}

\end{document}